\documentclass[10pt]{article}

\usepackage[letterpaper,top=0.5in, bottom=1in, left=0.8in, right=0.8in, headsep=0.in, centering]{geometry}
\usepackage[utf8]{inputenc}

\usepackage{amsmath}
\usepackage{amsfonts}
\usepackage{amssymb}
\usepackage{graphicx}
\usepackage{enumitem}
\usepackage{amsthm}
\usepackage{caption} 
\usepackage{array}
\usepackage{color}
\usepackage{kotex}
\usepackage{csquotes}
\usepackage{bookmark}
\usepackage{float}
\usepackage{multirow}

\MakeOuterQuote{"}

\makeatletter
\def\thickhline{%
  \noalign{\ifnum0=`}\fi\hrule \@height \thickarrayrulewidth \futurelet
   \reserved@a\@xthickhline}
\def\@xthickhline{\ifx\reserved@a\thickhline
               \vskip\doublerulesep
               \vskip-\thickarrayrulewidth
             \fi
      \ifnum0=`{\fi}}
\makeatother
\newlength{\thickarrayrulewidth}
\numberwithin{equation}{section}

\newcommand\spd[1][.5]{\mathbin{\vcenter{\hbox{\scalebox{#1}{$\wedge$}}}}}

\newcommand\innerprd[1][.5]{\mathbin{\vcenter{\hbox{\scalebox{#1}{$\bullet$}}}}}

\newcommand{\RN}[1]{%
  \textup{\uppercase\expandafter{\romannumeral#1}}%
}

\newcommand\Gnm[1]{\left\lfloor #1 \right\rceil}
\newcommand\btr[0]{\scalebox{1.2}{$\blacktriangle$}}
\newcommand\sgn[0]{\text{sgn}}
\newcommand\pht[1]{\text{pht}_{#1}}

\DeclareMathOperator*{\esssup}{ess\,sup}

\newtheorem{theorem}{Theorem}[section]
\newtheorem{lemma}[theorem]{Lemma}
\newtheorem{proposition}[theorem]{Proposition}

\theoremstyle{definition}
\newtheorem{definition}[theorem]{Definition}
\newtheorem{example}[theorem]{Example}
\newtheorem{remark}[theorem]{Remark}

\title{
The wave equation with specular derivatives}

\author{
Kiyuob Jung\thanks{Department of Mathematics, Kyungpook National University, Daegu, 41566, Republic of Korea, E-mail:~kyjung2357@knu.ac.kr},
Jehan Oh\thanks{Department of Mathematics, Kyungpook National University, Daegu, 41566, Republic of Korea, E-mail:~jehan.oh@knu.ac.kr}}

\begin{document}
\date{}
\maketitle

\begin{abstract}
In this paper, we construct the transport equation and the wave equation with specular derivatives and solve these equations in one-dimension. 
To solve these equations, we introduce new function spaces, which we term specular spaces, consisting of certain specularly differentiable functions. 

\end{abstract}

\smallskip


\smallskip


\section{Introduction}

The classical derivative can be generalized in many ways and generalized derivatives can be applied to deal with differential equations.
In this sense, a deeper analysis of generalized derivatives allows a deeper analysis of differential equations, including approximation, regularity, convergence, functional analysis, and so on.
This paper investigates a generalized derivative, so-called a specular derivative, and its application partial differential equations: the transport equation and the wave equation.

As for generalization of derivatives, \cite{1994_Bruckner} and \cite{1966_Bruckner} can be extensive and well-organized references.
We look over some generalized derivatives: symmetric derivatives, weak derivatives, subderivatives, and specular derivatives.
First, symmetric derivatives was first introduced by Aull \cite{1967_Aull}.
Quasi-Rolle's Theorem and Quasi-Mean Value Theorem for symmetric derivatives have preponderantly studied by some authors (see \cite{1983_Larson}, \cite{2011_Sahoo}, and \cite{1998_Sahoo}).
Calculations for symmetric derivatives are comfortable; symmetric derivatives can be calculated as the arithmetic mean of right and left derivatives.
However, symmetric differentiability allows some blowing up functions, including $f(x)=1/|x|$ for $x\in \mathbb{R}$.
Second, the concept of weak derivatives started from distributions (see \cite{2013_Bressan}) and is closely related with integrability (see \cite{2010_Evans}).
The Chain Rule, linearity, and other properties from classical derivatives still work in the weak derivative sense.
Third, geometry interpretation for subderivatives is intuitive and can be used to solve some partial differential equations (see \cite{2010_Evans}).
Finally, another new generalized derivative, which includes these advantages of generalized derivatives, was first introduced and scrutinized by Jung and Oh \cite{2022_Jung}.
In \cite{2022_Jung}, the second form of the Fundamental Theorem of Calculus, Quasi-Mean Value Theorem, tangent hyperplanes, and differential equations are addressed in light of specular derivatives.

A way to apply generalized derivatives is to construct and solve differential equations as in \cite{2013_Bressan} and \cite{2010_Evans}.
In particular, we refer to \cite{2011_Han_BOOK} for walk-through solving wave equations in the classical derivative sense.
In many cases, studying partial differential equations coincides with studying properties of a generalized derivative and its function spaces.  
Hence, we try to apply such approach to deal with the wave equation in light of specular derivatives.
In \cite{2022_Jung}, the special case of the transport equation with specular derivatives was only constructed and solved.
This paper aims to deepen this discussion to deal with the wave equation in the specular derivative sense. 

Now, we construct the Laplace's equation and the wave equation with specular derivatives.
However, we are interested in the latter.
Throughout this paper, $x = (x_1, x_2, \ldots, x_n) \in \mathbb{R}^{n}$ denotes a special variable, and $t\geq 0$ denotes a time variable.
For a function $u:\mathbb{R}^{n} \to \mathbb{R}$ with $u=u(x)$, the usual Laplacian $\triangle$ is defined on $u$ by 
\begin{equation*} 
  \triangle u = \sum_{i=1}^{n} \frac{\partial^2 u}{\partial x_i^2}
\end{equation*}
and we define the ($n$-\emph{dimensional}) \emph{Laplacian with specular derivatives} $\btr$ on $u$ by 
\begin{equation*} 
  \btr u := \sum_{i=1}^{n} \frac{\partial^2 u}{\partial^S x_i^2}.
\end{equation*}
For a function $u:\mathbb{R}^{n} \times \mathbb{R} \to \mathbb{R}$ with $u=u(x, t)$, the usual d'Alembertian $\square$ is defined on $u$ by 
\begin{equation*} 
  \square u = u_{tt} - \triangle u
\end{equation*}
and we also define the ($n$-\emph{dimensional}) \emph{d'Alembertian with specular derivatives} $\blacksquare$ on $u$ by 
\begin{equation*} 
  \blacksquare u := \partial^S_{tt} u - \btr u,
\end{equation*}
where 
\begin{equation*} 
  \partial^S_{tt} u := \frac{\partial^2 u}{\partial^S t^2}.
\end{equation*}
In this paper, we investigate the \emph{homogeneous wave equation with specular derivatives} (\emph{in $n$-dimensions})
\begin{equation} \label{PDE: homo wave eq. w/ spd} 
  \blacksquare u = 0
\end{equation}
and the \emph{nonhomogeneous wave equation with specular derivatives} (\emph{in $n$-dimensions})
\begin{equation} \label{PDE: nonhomo wave eq. w/ spd}
  \blacksquare u = f,
\end{equation}
subject to appropriate initial and boundary conditions, where $\Omega$ is an open set in $\mathbb{R}^{n}$, the function $u : \overline{\Omega} \times [0, \infty) \to \mathbb{R}$ with $u=u(x, t)$ is unknown and the function $f:\Omega \times [0, \infty) \to \mathbb{R}$, called the \emph{force} (\emph{term}) of the equation, is given.
Some authors call the source the \emph{nonhomogenoues term} or the \emph{source} (\emph{term}).

Here are our main results.
The first form of the fundamental Theorem of Calculus with specular derivatives is stated and proved. 
We define a function space, which we call a \emph{specular space}, consisting of specularly differentiable functions satisfying two certain conditions.
As for regularity, classical differentiability and specular differentiability are related inductively.
We prove that the specular gradient analogously is related with the vector perpendicular to the surface of a function if the existence of the strong specular tangent hyperplane is guaranteed.
Green's Theorem with specular derivatives in two-dimensions is stated and proved. 
As for differential equations, we construct the transport equation and the wave equation with specular derivatives. 
We only address with above equations in one-dimension and solve homogeneous transport equation, homogeneous wave equation, and nonhomogeneous wave equation in infinite domain.
Initial conditions and a boundary condition are considered.
All solutions are equal with that of differential equations with classical derivatives. 

The rest of the paper is organized as follows.
In Section 2, we recall and the specular derivative's properties stated in \cite{2022_Jung}.
Add to this we state and prove the first form of the Fundamental Theorem of Calculus with specular derivatives.
Also, we extend the concept of piecewise continuity into high-dimensional space $\mathbb{R}^{n}$.
In Section 3, we define the specular space collecting certain specular derivatives and scrutinize its properties. 
The regularity of the specular space in terms of continuous function spaces is stated and proved.
In Section 4, we construct the one-dimensional transport equation and the wave equation in light of specular derivatives and solve them in infinite domain.

\section{Preliminaries}

Suppose $\Omega$ is an open set in $\mathbb{R}^{n}$ and $u:\Omega \to \mathbb{R}$ is a function.
Let $i \in \mathbb{N}$ be an index with $1 \leq i \leq n$.
Denote $e_i$ be the $i$-th standard basis vector of $\mathbb{R}^{n}$.
The (first order) right and left partial derivative of $u$ at $x$ with respect to $x_i$ are defined as the limits 
\begin{equation*} 
  \partial_{x_{i}}^+ u(x):=\lim _{h \searrow 0} \frac{u\left(x + h e_i\right)-u(x)}{h} 
  \qquad \text{and} \qquad
  \partial_{x_{i}}^- u(x):=\lim _{h \nearrow 0} \frac{u\left(x + h e_i\right)-u(x)}{h}
\end{equation*}
as a real number, respectively.
Write 
\begin{equation*} 
  u[x)_{(i)} := \lim _{h \searrow 0} u(x + h e_i) 
  \qquad \text{and} \qquad
  u(x]_{(i)} := \lim _{h \nearrow 0} u(x + h e_i)
\end{equation*}
if each limit exists.
Defining $u[x]_{(i)} := \frac{1}{2}\left(u[x)_{(i)} + u(x]_{(i)} \right)$, write 
\begin{equation*} 
    \overline{x}_{(i)} := (a, u[x]_{(i)}),
\end{equation*}
where $1 \leq i \leq n$.
In particular, if $u[x]_{(1)}=u[x]_{(2)}=\cdots=u[x]_{(n)}$, we write the common value as $u[x]$.

The (first order) right and left specularly partial derivative of $u$ at $x$ with respect to $x_i$ are defined as the limits 
\begin{equation*} 
  \partial_{x_{i}}^{R} u(x):=\lim _{h \searrow 0} \frac{u\left(x + h e_i\right)-u[x)_{(i)}}{h} 
  \qquad \text{and} \qquad
  \partial_{x_{i}}^{L} u(x):=\lim _{h \nearrow 0} \frac{u\left(x + h e_i\right)-u(x]_{(i)}}{h}
\end{equation*}
as a real number, respectively.
In particular, $u$ is (first order) semi-specularly partial differentiable at $x$ if there exist both $\partial_{x_{i}}^{R} u(x)$ and $\partial_{x_{i}}^{L} u(x)$.
If $n=1$, we write $\partial_{x_{i}}^{R}u$ and $\partial_{x_{i}}^{L}u$ as $u^{\spd}_+$ and $u^{\spd}_-$ and call these the right and left specular derivatives, respectively.
Also, $u$ is (first order) semi-specularly differentiable at $x$ if there exist both $u^{\spd}_+(x)$ and $u^{\spd}_-(x)$.

The section $\Omega_{x_{i}}(x)$ of $\Omega$ of $u$ by $x$ with respect to $x_i$ is the set 
\begin{equation*} 
  \Omega_{x_{i}}(x) =\left\{y \in \Omega : y  e_j = x  e_j \text { for all } 1 \leq j \leq n \text { with } j \neq i\right\}.
\end{equation*}
If $u$ is semi-specularly partial differentiable at $x$ with respect to $x_i$, a phototangent $\pht{x_i} u:\mathbb{R}^{n}_{x_i}(x) \to \mathbb{R}$ of $u$ at $x$ with respect to $x_i$ is defined by 
\begin{equation*}
\pht{x_i} u(y)= 
\begin{cases}
  \partial_{x_{i}}^{L} u(x)(y  e_i-x  e_i) + u(x]_{(i)} 
  & \text {if } y  e_i<x  e_i, \\ 
  u[x]_{(i)} 
  & \text {if } y  e_i=x  e_i, \\ 
  \partial_{x_{i}}^{R} u(x)(y  e_i-x  e_i) + u[x)_{(i)} 
  & \text {if } y  e_i>x  e_i,
\end{cases}
\end{equation*}
for $y \in \mathbb{R}^{n}_{x_i}(x)$.
If $n=1$, we denote the phototangent briefly by $\pht{} u = \pht{} u(y)$.

Write $\mathcal{V}(u, x)$ to be the set containing all indices $i$ of variables $x_i$ such that $u$ is semi-specularly partial differentiable at $x$ with respect to $x_i$ for each $1 \leq i \leq n$, and $\mathcal{P}(u, x)$ to be the set containing all intersection points of the phototangent of $u$ at $x$ with respect to $x_i$ and a sphere $\partial B(\overline{x}_{(i)}, 1)$ for each $1 \leq i \leq n$.
If $|\mathcal{P}(u, x)| \geq n+1$ and $u[x]_{(i)}=u[x]_{(j)}$ for all $i, j \in \mathcal{V}(u, x)$, we say that $u$ is weakly specularly differentiable at $x$; if $|\mathcal{P}(u, x)|=2 n$ and $u[x]_{(1)}=u[x]_{(2)}=\cdots=u[x]_{(n)}$, we say that $u$ is (strongly) specularly differentiable at $x$.

Recall the following useful formulas to calculate a specular derivatives:
  \begin{enumerate}[label=(F\arabic*), ref=(F\arabic*), start=1]
    \rm\item\label{F1} 
      $\partial^S_{x_i}u(x) =
        \begin{cases} 
          \displaystyle \frac{\alpha_i \beta_i -1 + \sqrt{\left(\alpha_i^2 + 1\right)\left( \beta_i^2 + 1 \right)}}{\alpha_i + \beta_i} & \text{if } \alpha_i + \beta_i \neq 0,\\
          0 & \text{if } \alpha_i + \beta_i = 0;
        \end{cases}
      $
    \rm\item\label{F2} $\displaystyle \partial^S_{x_i}u(x) = \tan\left( \frac{\arctan \alpha_i + \arctan \beta_i}{2} \right)$,
  \end{enumerate}
  where $\alpha_i = \partial_{x_{i}}^{R} u(x)$ and $\beta_i = \partial_{x_{i}}^{L} u(x)$.
  In this paper, we temporarily employ the notation in the formula $\text{\ref{F1}}$:
  \begin{equation} \label{A notation} 
    A(\alpha, \beta) := \frac{\alpha \beta - 1 + \sqrt{(\alpha^2 + 1)(\beta^2 + 1)}}{\alpha + \beta}
  \end{equation}
  for $\alpha$, $\beta \in \mathbb{R}$ with $\alpha + \beta \neq 0$.
  Note that $A(\alpha, \beta) = A(\beta, \alpha)$ for any $\alpha$, $\beta \in \mathbb{R}$ with $\alpha + \beta \neq 0$.

\begin{remark} \label{Rmk : specular derivative with minus variable}
  A simple application of \ref{F1} yields that 
  \begin{equation*} 
    \partial^{S}_{x_i}u(x_1, \ldots, -x_i, \ldots, x_n) = - \partial^{S}_{x_i}u(x)
  \end{equation*}
  for each $1 \leq i \leq n$.
\end{remark}

In the classical derivative sense, the Fundamental Theorem of Calculus (FTC for short) can be sorted into the first form and the second form (see \cite{2011_Bartle_BOOK}).
In \cite{2022_Jung}, only the second form of FTC with specular derivatives was stated.

\begin{theorem} \label{Thm : the second form of FTC with spd}
  \emph{(The second form of the Fundamental Theorem of Calculus with specular derivatives)}
  Suppose $f : [a, b] \to \mathbb{R}$ be a piecewise continuous function.
  Assume 
  \begin{equation} \label{Eq : the second form of FTC with spd}
    f(x) =
    \begin{cases} 
    A(f(x], f[x))& \text{if } f(x] + f[x) \neq 0,\\ 
    0 & \text{if } f(x] + f[x) = 0,
    \end{cases}
  \end{equation}
  for each point $x \in (a, b)$.
  Let $F$ be the indefinite integral of $f$.
  Then the following properties hold:
  \begin{enumerate}[label=(\roman*)] 
  \rm\item $F$ is continuous on $[a, b]$.
  \rm\item $F^{\spd}(x) = f(x)$ for all $x \in [a, b]$.
  \end{enumerate}
\end{theorem}

Note that to guarantee that $F^{\spd}$ equal with $f$ the condition \eqref{Eq : the second form of FTC with spd} is necessary.
However, the condition \eqref{Eq : the second form of FTC with spd} can be dropped in calculating the integration of $f$.
Here, we provide the first form of FTC with specular derivatives and its proof. 

\begin{theorem} \label{Thm : the first form of FTC with spd}
  \emph{(The first form of the Fundamental Theorem of Calculus with specular derivatives)}
  Suppose $f:[a, b] \to \mathbb{R}$ be a piecewise continuous function with a singular set $\mathcal{S}$.
  If there exists $F:[a, b] \to \mathbb{R}$ such that:
  \begin{enumerate}[label=(\roman*)] 
  \rm\item $F$ is continuous on $[a, b]$.
  \rm\item $F^{\spd}(x) = f(x)$ for all $x \in [a, b] \setminus \mathcal{S}$.
  \end{enumerate}
  Then 
  \begin{equation*} 
    \int_a^b f(x) ~dx= F(b) - F(a).
  \end{equation*}  
\end{theorem}

\begin{proof} 
  Let $s_1$, $s_2$, $\ldots$, $s_n$ be singular points in $\mathcal{S}$, where $a < s_1 < s_2 < \ldots < s_n < b$.
  Also, since $f$ is continuous on $[a, b] \setminus \mathcal{S}$, we have 
  \begin{equation*} 
    f(x) = F^{\spd}(x) = F'(x)
  \end{equation*}
  for all $x \in [a, b] \setminus \mathcal{S}$.
  Writing $s_0 := a$ and $s_{n+1} := b$, define the function $\overline{f_i}:[s_i, s_{i + 1}] \to \mathbb{R}$ by 
  \begin{equation*}
  \overline{f_i}(x) =
  \begin{cases} 
  f(x] & \text{if } x = s_{i + 1},\\ 
  f(x) & \text{if } s_i < x < s_{i + 1},\\ 
  f[x) & \text{if } x = s_i, 
  \end{cases}
  \end{equation*}
  for each $i=0, 1, \ldots, n$; note that $\overline{f_i}$ is continuous on $[s_i, s_{i + 1}]$.
  Thanks to the first form of FTC with classical derivatives, we find that 
  \begin{align*}
    \int_a^b f(x) ~dx 
    &= \sum_{i=0}^n \int_{s_i}^{s_{i + 1}} f(x) ~dx \\
    &= \sum_{i=0}^n \int_{s_i}^{s_{i + 1}} \overline{f_i}(x) ~dx \\
    &= \sum_{i=0}^n [F(s_{i + 1}) - F(s_i)] \\
    &= F(b) - F(a),
  \end{align*}
  as required.
\end{proof}

Since we can apply FTC with specular derivatives to a piecewise continuous function, we assume that integrands are piecewise continuous. 
To clarify the piecewise continuity in higher dimensions, we generalize a jump discontinuity in higher dimensions by using the concept of affine hyperplanes.
An affine hyperplane $\hbar$ in a set $X \subset \mathbb{R}^{n}$ is defined by
\begin{equation*} 
  a \innerprd x = b,
\end{equation*}
where $x \in X$ is a variable vector and a non-zero constant vector $a\in \mathbb{R}^{n}$ and a scalar $b\in \mathbb{R}$ are given.
The set 
\begin{equation*} 
  \mathcal{G}(\hbar) := \left\{ x \in X : a \innerprd x = b \right\}
\end{equation*}
is the graph of the affine hyperplane $\hbar$.

\begin{definition}
  Let $\Omega$ be an open set in $\mathbb{R}^n$ and $u:\Omega \to \mathbb{R}$ be a multi-variable function.
  If there exists a hyperplane $\hbar$ in a connected subset of $\Omega$ such that for a variable $x_i$
  \begin{equation*} 
    u(x]_{(i)} \neq u[x)_{(i)}    
  \end{equation*}
  for all $x \in \mathcal{G}(\hbar)$, then we say $u$ is \emph{jump discontinuous} with respect to $x_i$ on $\mathcal{G}(\hbar)$ and call $\hbar$ a \emph{singular hyperplane} of $u$ in $\Omega$.
\end{definition}

This is the generalization of the concept of the usual jump discontinuity in one-dimension.
Indeed, if $n=1$, a singular hyperplane is a point which is equal with its graph.

In two-dimensions, then we also call singular hyperplanes \emph{singular lines} can be written by $a_1x_1 + a_2x_2 = b$.
Also, singular lines with either $a_1 = 0$ or $a_2 = 0$ are parallel to either $x_1$-axis or $x_2$-axis, respectively, while singular lines with $a_1 \neq 0$ and $a_2 \neq 0$ are not parallel to any axises.

Now, we inductively define the concept of piecewise continuity in higher dimensions.

\begin{definition}
  Let $u:\Omega \to \mathbb{R}$ be a function on an open set $\Omega \subset \mathbb{R}^{n}$.
  We say that $u$ is ($n$-\emph{dimensional}) \emph{piecewise continuous} if there exist $m \geq 0$ singular hyperplanes $\hbar_1$, $\hbar_2$, $\ldots$, $\hbar_m$ of $u$ in $\Omega$ satisfying the following conditions:
  \begin{enumerate}[label=(\roman*)] 
  \rm\item $u$ is continuous in $\Omega \setminus \left[ \mathcal{G}(\hbar_1) \cup \mathcal{G}(\hbar_2) \cup \cdots \cup \mathcal{G}(\hbar_m) \right]$;
  \rm\item for each $\ell \in \left\{ 1, 2, \ldots, m \right\}$, the restriction of $u$ to $\mathcal{G}(\hbar_{\ell})$ is continuous on $\mathcal{G}(\hbar_{\ell}) \setminus \mathcal{G}^{\ast}(\hbar_{\ell})$, where 
  \begin{equation*} 
    \mathcal{G}^{\ast}(\hbar_{\ell}) := \mathcal{G}(\hbar_1) \cup \cdots \cup \mathcal{G}(\hbar_{\ell - 1}) \cup \mathcal{G}(\hbar_{\ell + 1}) \cup \cdots \cup \mathcal{G}(\hbar_m).
  \end{equation*}
  \end{enumerate}
\end{definition}

\noindent The definition of $n$-dimensional piecewise continuity is generalization of piecewise continuity in the sense that a singular hyperplane in one-dimension is a singular point at which a function has jump discontinuity.
Note that a continuous function can be considered as a piecewise continuous function with no singular hyperplanes. 

\begin{remark}
  If a multi-variable function $u$ is $n$-dimensional piecewise continuous in $\Omega$, then $u$ is continuous a.e. in $\Omega$ since each graph of singular hyperplanes has measure zero.
  In the following Figure \ref{Fig : Piecewise continuity in one, two, and three dimensions}, each red figure illustrates a singular hyperplanes in one, two, and three dimension.
\end{remark}

\begin{figure}[H] 
  \centering 
  \includegraphics[width=0.9\textwidth]{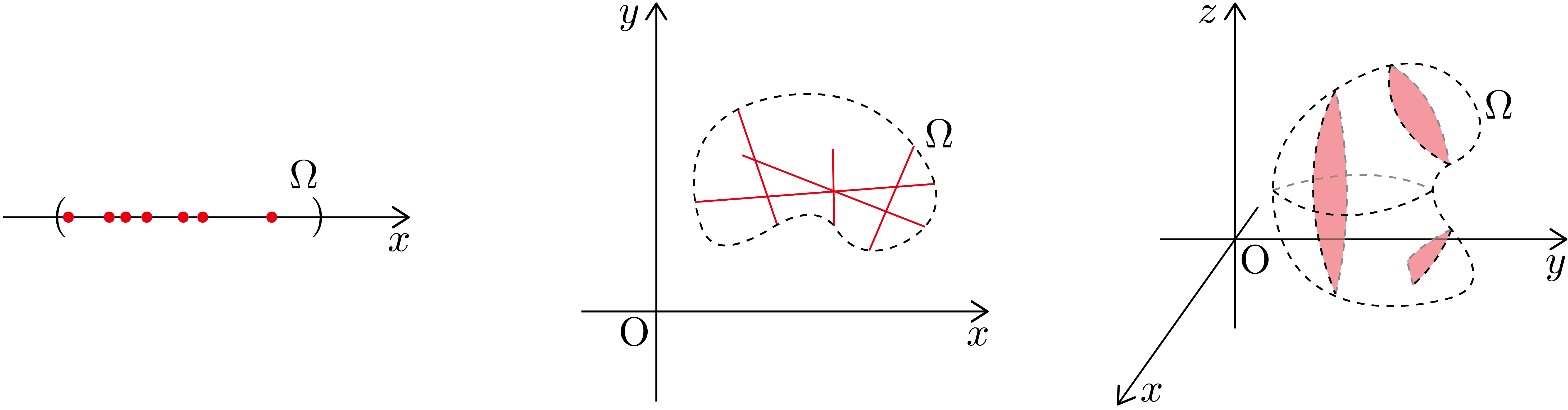} 
  \caption{Piecewise continuity in one, two, and three dimensions}
  \label{Fig : Piecewise continuity in one, two, and three dimensions}
\end{figure}

\begin{example}
  Consider the function $u:\mathbb{R}^{2} \to \mathbb{R}$ defined by 
  \begin{equation} \label{Ex : n-dim p.w.c. function}
    u(x, y) = |2x - y| + |x - 3| =
    \begin{cases} 
     3x - y - 3  & \text{if } 2x - y \geq 0, x - 3 \geq 0,\\ 
     x - y + 3   & \text{if } 2x - y \geq 0, x - 3 < 0,\\ 
     - x + y - 3 & \text{if } 2x - y < 0, x - 3 \geq 0,\\ 
     -3x + y + 3 & \text{if } 2x - y < 0, x - 3 < 0, 
    \end{cases}   
  \end{equation}
  for $(x, y) \in \mathbb{R}^{2}$.
  Then simple calculation yields that 
  \begin{equation*}
    \partial^S_{x} u(x, y)  =
    \begin{cases} 
     3  & \text{if } 2x - y > 0, x - 3 > 0,\\ 
     A(3, 1)  & \text{if } 2x - y > 0, x - 3 = 0,\\ 
     1   & \text{if } 2x - y > 0, x - 3 < 0,\\ 
     A(3, -1)  & \text{if } 2x - y = 0, x - 3 > 0,\\ 
     A(3, -3)  & \text{if } 2x - y = 0, x - 3 = 0,\\ 
     A(1, -3) & \text{if } 2x - y = 0, x - 3 < 0,\\ 
     - 1 & \text{if } 2x - y < 0, x - 3 > 0,\\ 
     A(-1, -3) & \text{if } 2x - y < 0, x - 3 = 0,\\ 
     -3 & \text{if } 2x - y < 0, x - 3 < 0,
    \end{cases}
  \end{equation*}
  using the notation \eqref{A notation}.
  Choose two singular lines $\hbar_1$ and $\hbar_2$: 
  \begin{align*}
    \hbar_1 &:  2x - y = 0, \\
    \hbar_2 &:  x - 3 = 0
  \end{align*}
  for $(x, y) \in \mathbb{R}^{2}$.
  Finally, one can find that $\partial^S_{x} u$ is two-dimensional piecewise continuous in $\mathbb{R}^{2}$.
\end{example}

Now, we can define what a proper integrand is in solving the wave equation with specular derivatives.

\begin{definition} \label{Def : proper integrands}
  Let $u:\Omega \to \mathbb{R}$ be a function on an open set $\Omega \subset \mathbb{R}^{n}$.
  We say that $u$ is \emph{proper} if $u$ satisfies the following conditions:
  \begin{enumerate}[label=(\roman*)] 
  \rm\item $u$ is $n$-dimensional piecewise continuous in $\Omega$;
  \rm\item given a point $a \in \Omega$,
  \begin{equation*}
    u(a) =
  \begin{cases} 
    A\left( u(a]_{(i)}, u[a)_{(i)} \right) & \text{if } u(a]_{(i)} + u[a)_{(i)} \neq 0,\\ 
    0 & \text{if } u(a]_{(i)} + u[a)_{(i)} = 0,
  \end{cases}
  \end{equation*}
  for all $1 \leq i \leq n$.
  \end{enumerate}
  Also, we denote $S^0(\Omega)$ as the set of all proper function $u : \Omega \to \mathbb{R}$.
\end{definition}

\noindent Note that the above second condition stems from the condition \eqref{Eq : the second form of FTC with spd}.
Also, it is clear that a continuous function is proper. 

\begin{remark}
  One can apply FTC with specular derivatives to functions in $S^0$.  
\end{remark}

Since solutions we will drive after includes integrands, we will assume that  wave’s vertical velocity in initial value problems and a given force in nonhomogeneous problems are proper.

\section{The specular space}

Before investigating into the wave equation, we define a function space of specular derivatives and state basic properties of specular derivatives for functional analysis.

To efficiently denote higher-order specular derivatives of a function, we employ a multi-index 
\begin{equation*} 
  \alpha = ( \alpha_1, \alpha_2, \cdots, \alpha_n),
\end{equation*}
where each component $\alpha_i$ is a nonnegative integer for $1 \leq i \leq n$, with order
\begin{equation*} 
  \left\vert \alpha \right\vert = \alpha_1 + \alpha_2 + \cdots + \alpha_n.
\end{equation*}
Here, we suggest the notations for higher-order specular derivatives with multi-index as follows.

\begin{definition}
  Let $u:\Omega \to \mathbb{R}$ be a real-valued function on an open set $\Omega$ in $\mathbb{R}^{n}$ and $x$ be a point in $\Omega$.
  For convenience, write $\left( \partial^S_{x_i} \right)^0 u = u$ for $1 \leq i \leq n$.
  \begin{enumerate}[label=(\alph*)] 
  \rm\item Given multi-index $\alpha$, denote a specularly differential operator of order $\left\vert \alpha \right\vert$ by 
  \begin{equation*} 
    \left(\partial^S\right)^{\alpha} u(x) := \frac{\partial^{\left\vert \alpha \right\vert} u}{\partial^S x_1^{\alpha_1} \partial^S x_2^{\alpha_2} \cdots \partial^S x_n^{\alpha_n}} (x) = \left(\partial^S_{x_1}\right)^{\alpha_1} \left(\partial^S_{x_2}\right)^{\alpha_2} \cdots \left(\partial^S_{x_n}\right)^{\alpha_n} u(x).
  \end{equation*}
  Note that $\left( \partial^S \right)^{(0, 0, \cdots, 0)} u = u$.
  In particular, if $|\alpha| = 2$, we simply write 
  \begin{equation*} 
    \partial^S_{x_i x_j} u := \partial^S_{x_j} \partial^S_{x_i} u
    \qquad \text{and} \qquad
    \partial^S_{x_i x_i} u := \left( \partial^S_{x_i} \right)^{2} u
  \end{equation*}
  for $1 \leq i, j \leq n$.
  \rm\item If $k$ is a nonnegative integer, denote the set of all specular partial derivatives of order $k$ by 
  \begin{equation*} 
    \left( \partial^S \right)^k u(x) := \left\{ \left(\partial^S\right)^{\alpha} u(x) : \left\vert \alpha \right\vert = k \right\}.
  \end{equation*}
  \end{enumerate}
\end{definition}

\subsection{The specular space}

In this subsection, we define the function space for specular derivatives and state its properties.

\begin{definition} \label{Def : the specular space}
  Let $\Omega$ be an open subset of $\mathbb{R}^{n}$ and let $k$ be a natural number.
  Let $\alpha$ be a multi-index with order $|\alpha| \geq 0$.
  We define the \emph{specular space} $S^{k}(\Omega)$ to be the set of all continuous functions $u:\Omega \to \mathbb{R}$ satisfying the following conditions:
  \begin{enumerate}[label=(\roman*)] 
  \rm\item if $| \alpha | \leq k$, then $\left( \partial^S \right)^{\alpha} u$ exists and is in $S^0(\Omega)$; \label{cond. 1}
  \rm\item if $| \alpha | \leq k - 2$, then $\partial^S_{x_i x_j}\left( \partial^S \right)^{\alpha} u$ is continuous in $\Omega$ whenever $i \neq j$. \label{cond. 2}
  \end{enumerate}
\end{definition}

As for the above definition, the piecewise continuity in the condition \ref{cond. 1} eliminates counterexamples such as function in $C^1$-class but not in $C^2$-class (see Example \ref{Ex : C^1 but not in C^2}); 
the condition \ref{cond. 2} is a crucial in proving the symmetry of second specular derivatives (see Theorem \ref{Thm : Symmetry of second specular derivatives in 2d}).

Before we prove the specular space is a vector space, the following lemma is needed.

\begin{lemma} \label{Lem : closedness of specularly differentiablity}
  If $u$, $v$ are specularly differentiable on an open set $\Omega \subset \mathbb{R}$, then $au+bv$ is specularly differentiable on $\Omega$ for all $a, b \in \mathbb{R}$.
\end{lemma}

\begin{proof} 
  Let $x \in \Omega$.
  Since $u$ and $v$ are specularly differentiable at $x$, the phototangents $\pht{} u$ and $\pht{} v$ are both continuous at $x$ by Proposition 2.6. 
  Let $y \in \mathbb{R}$.
  If $y = x$, one can find that 
  \begin{align*}
    \pht{} (au + bv)(y) 
    &= (au + bv)[y] \\
    &= a u[y] + b v[y] \\
    &= a \pht{} u(y) + b \pht{} v(y).
  \end{align*}
  Next, if $y > x$, we see that 
  \begin{align*}
    \pht{} (au + bv)(y) 
    &= (au + bv)^{\spd}_+(x)(y - x) + (au + bv)[x) \\
    &= \left[ au^{\spd}_+(x) + bv^{\spd}_+(x) \right](y - x) + au[x) + bv[x) \\
    &= a\left[ u^{\spd}_+(x) + u[x) \right] + b\left[ v^{\spd}_+(x) + v[x) \right] \\
    &= a \pht{} u(y) + b \pht{} v(y).
  \end{align*}
  Similarly, $\pht{} (au + bv)(y) = a \pht{} u(y) + b \pht{} v(y)$ whenever $y < x$.
  Since $a \pht{} u + b \pht{} v$ is continuous at $x$, $\pht{} (au + bv)$ is continuous at $x$, completing the proof.
\end{proof}

\begin{proposition}
  The specular space $S^{k}$ is a vector space whenever $k$ is a nonnegative integer.
\end{proposition}

\begin{proof} 
  Let $\Omega$ be is an open set in $\mathbb{R}^{n}$ and let $V(\Omega)$ denote the vector space consisting of all functions $f:\Omega \to \mathbb{R}$.
  Let $a$, $b \in \mathbb{R}$ and let $u$, $v \in S^{k}(\Omega)$.
  We claim that $au + bv \in S^{k}(\Omega)$.
  Fix $i \in \left\{ 1, 2, \ldots, n \right\}$ and $x_0 \in \Omega$.
  It is clear that $a \partial^S_{x_i} u + b \partial^S_{x_i} v$ is continuous a.e. in $\Omega$.
  By Lemma \ref{Lem : closedness of specularly differentiablity}, we have $a \partial^S_{x_i} u + b \partial^S_{x_i} v$ is specularly partial differentiable.
  Hence, we conclude that $S^{k}(\Omega)$ is the subspace of $V(\Omega)$ and then is a vector space.
\end{proof}

Here, we provide the regularity of specular spaces.

\begin{theorem} \label{Thm : regularity of specular space}
  For an open set $\Omega \subset \mathbb{R}^{n}$, we have 
  \begin{equation*} 
    S^{k+1}(\Omega) \subset C^k(\Omega) \subset S^k(\Omega)
  \end{equation*}
  for each $k \in \mathbb{N}$.
\end{theorem}

\begin{proof} 
  Since the classical differentiability implies the specularly differentiability, it holds $C^1(\Omega) \subset S^1(\Omega)$. 
  Next, let $1 \leq i \leq n$. 
  Then there exists $u_{x_i}$ on $\Omega$ due to the fact that the second order specular differentiability implies the first order classical differentiability, see \cite[Theorem 2.27]{2022_Jung}. % MERGE
  Moreover, by \cite[Proposition 2.26]{2022_Jung}, $u_{x_i}$ is continuous on $\Omega$ since $u$ is continuous on $\Omega$ and $\partial^S_{x_ix_i}u$ exists.  % MERGE
  Hence, we have $S^2(\Omega) \subset C^1(\Omega)$.
  Finally, the mathematical induction completes the proof.
\end{proof}

\begin{example}
  In general, $C^k(\Omega) \subset S^{k+1}(\Omega)$ does not hold for a nonnegative integer $k$.
  In other words, the continuity of $u_{x_i}$ does not ensure the existence of $\partial^S_{x_i}u_{x_i}$.
  Take the counterexample as the function $u(x) = x \sin\left( \frac{1}{x} \right)$, which is continuous at $x=0$ but right-hand and left-hand derivatives do not exist.
\end{example}

\begin{example}
  The function defined as in \eqref{Ex : n-dim p.w.c. function} is in the class $S^1\left(\mathbb{R}^{2}\right)$.
\end{example}

\begin{remark}
  Let $\Omega$ be an open set in $\mathbb{R}^2$ and $(x, y)$ be a point in $\mathbb{R}^2$.
  Thanks to Theorem \ref{Thm : regularity of specular space}, the definition of the specular space in two-dimensions can be rewritten as follows.
  The set $S^2(\Omega)$ consists of all continuous functions $u:\Omega \to \mathbb{R}$ satisfying the following conditions:
    \begin{enumerate}[label=(\roman*)] 
    \rm\item $u_x$, $u_y$, $\partial^S_x u_x$, $\partial^S_y u_x$, $\partial^S_x u_y$, and $\partial^S_y u_y$ exist and are in $S^0(\Omega)$;
    \rm\item $\partial^S_{xy}u$ and $\partial^S_{yx}u$ are continuous in $\Omega$. 
    \end{enumerate}
  Moreover, given a function $u \in S^2(\Omega)$, Theorem \ref{Thm : regularity of specular space} allows us to simply write as follows:
  \begin{equation*} 
    \partial^S_{x_ix_j}u = \frac{\partial^2 u}{\partial^S x_i \partial^S x_j} = \frac{\partial^2 u}{\partial^S x_i \partial x_j} = \partial^S_{x_i}u_{x_j}
  \end{equation*}
  for every $1 \leq i, j \leq 2$.
\end{remark}

\begin{example} \label{Ex : C^1 but not in C^2}
  Consider the function $f:\mathbb{R}^{2} \to \mathbb{R}$ defined by 
  \begin{equation*}
  f(x, y) =
  \begin{cases} 
    \displaystyle \frac{xy(x^2 - y^2)}{x^2 + y^2} & \text{if } (x, y) \neq (0, 0),\\[0.1cm]
    0 & \text{if } (x, y) = (0, 0),
  \end{cases}
  \end{equation*}
  which is a typical example showing the existence of a function in $C^1$-class but not in $C^2$-class.
  This function $f$ is not in $S^2$-class as $f$ does not have any singular line. 
\end{example}

\begin{example}
  Consider the function $q : \mathbb{R} \to \mathbb{R}$ by 
  \begin{equation} \label{Fnc : q}
    q(x) = \frac{1}{2}x|x|
  \end{equation}
  for $x \in \mathbb{R}$.
  Clearly, $q \in C^1(\mathbb{R})$.
  Moreover, one can calculate that 
  \begin{equation*} 
    q'(x) = x \sgn(x) = |x|   
    \qquad \text{and} \qquad 
    q^{\spd\spd}(x) = \sgn(x) = \frac{|x|}{x} = \frac{x}{|x|}.
  \end{equation*}
  Since the sign function has the singular point $x=0$, we conclude $q \notin C^2(\mathbb{R})$  but $q \in S^{2}(\mathbb{R})$.
\end{example}

\begin{remark}
  In two-dimensions, $C^2$ is a proper subset of $S^2$, i.e., 
  \begin{equation*} 
    C^2(\Omega) \subsetneq S^2(\Omega)
  \end{equation*}
  for an open set $\Omega \subset \mathbb{R}^{2}$.
  Take a counterexample as the function $u:\mathbb{R}^{2} \to \mathbb{R}$ defined by 
  \begin{equation*} 
    u(x, y) = q(x) + q(y)
  \end{equation*}
  for $(x, y) \in \mathbb{R}^{2}$, where $q$ is defined as in \eqref{Fnc : q}.
  Then the function $u$ is $S^2$-function but is not twice specularly differentiable with respect to $x$ as well as $y$.
\end{remark}

Now, we prove the symmetry of second specular derivatives in two-dimensions.
Recall that the condition \ref{cond. 2} in Definition \ref{Def : the specular space} implies that $\partial^S_x u_y$ and $\partial^S_y u_x$ are continuous functions if $u$ is in $S^2$-class.

\begin{theorem} \label{Thm : Symmetry of second specular derivatives in 2d}
  Let $\Omega$ be an open set in $\mathbb{R}^2$ and $(x, y)$ be a point in $\mathbb{R}^2$.
  If $u \in S^2(\Omega)$, then 
  \begin{equation*} 
    \frac{\partial^2 u}{\partial^S x \partial^S y} = \frac{\partial^2 u}{\partial^S y \partial^S x}.
  \end{equation*}
\end{theorem}

\begin{proof} 
  Let $(a, b)$ be a point in $\Omega$ and choose sufficiently small $\Delta x > 0$ and $\Delta y > 0$.
  Consider a rectangle $R$ with vertices $(a, b)$, $(a + \Delta x, b)$, $(a, b + \Delta y)$, and $(a + \Delta x, b + \Delta y)$.
  Define a single variable function $F$ by
  \begin{equation*} 
    F(x) := u(x, b + \Delta y) - u(x, b)
  \end{equation*}
  and a two variable function $E$ by
  \begin{equation*} 
    E(\Delta x, \Delta y) := u(a + \Delta x, b + \Delta y) - u(a + \Delta x, b) - u(a, b + \Delta y) + u(a, b).
  \end{equation*}
  Applying the Mean Value Theorem to $F$, there exists $c \in (a, a + \Delta x)$ such that 
  \begin{equation} \label{Thm : Symmetry of second specular derivatives in 2d - 1}
    \frac{F(a + \Delta x) - F(a)}{\Delta x} = F'(c).
  \end{equation}
  Also, one can calculate that 
  \begin{equation} \label{Thm : Symmetry of second specular derivatives in 2d - 2}
    \left.\frac{d}{dx}\right\vert_{x=c} F(x) = u_x (c, b + \Delta y) - u_x(c, b).
  \end{equation}
  Combining \eqref{Thm : Symmetry of second specular derivatives in 2d - 1} and \eqref{Thm : Symmetry of second specular derivatives in 2d - 2}, we have 
  \begin{align}  
    E(\Delta x, \Delta y) &= F(a + \Delta x) - F(a)  \nonumber \\
    &= F'(c) \Delta x \nonumber \\
    &= \left[ u_x (c, b + \Delta y) - u_x(c, b) \right] \Delta x. \label{Thm : Symmetry of second specular derivatives in 2d - 3}
  \end{align}
  Since $u_x$ is continuous and specularly partial differentiable with respect to $y$, the Quasi-Mean Value Theorem implies that there exist $d_1$, $d_2 \in (b, b + \Delta y)$ such that 
  \begin{equation*} 
    \partial^S_y u_x (c, d_1) \leq \frac{u_x(c, b + \Delta y) - u_x(c, b)}{\Delta y} \leq \partial^S_y u_x (c, d_2).
  \end{equation*}
  Owing to \eqref{Thm : Symmetry of second specular derivatives in 2d - 3}, we obtain that 
  \begin{equation} \label{Thm : Symmetry of second specular derivatives in 2d - 4}
    \partial^S_y u_x (c, d_1) \leq \frac{E(\Delta x, \Delta y)}{\Delta x \Delta y} \leq \partial^S_y u_x (c, d_2).
  \end{equation}
  Since $\partial^S_y u_x $ is assumed to be continuous, 
  \begin{equation*} 
    \lim_{(\Delta x, \Delta y) \to (0, 0)} \partial^S_y u_x(c, d_1)  = \partial^S_y u_x(a, b) = \lim_{(\Delta x, \Delta y) \to (0, 0)} \partial^S_y u_x(c, d_2) .
  \end{equation*}
  Taking the limit $(\Delta x, \Delta y) \to (0, 0)$ to \eqref{Thm : Symmetry of second specular derivatives in 2d - 4}, the squeeze theorem conclude that 
  \begin{equation*} 
    \partial^S_y u_x(a, b) = \lim_{(\Delta x, \Delta y) \to (0, 0)} \frac{E(\Delta x, \Delta y)}{\Delta x \Delta y}.
  \end{equation*}

  Next, define a single variable function $G$ by
  \begin{equation*} 
    G(y) = u(a + \Delta x, y) - u(a, y).
  \end{equation*}
  As before, we can apply the Mean Value Theorem and the Quasi-Mean Value Theorem to $G$, finding that there exist $d \in (b, b + \Delta y)$ and $c_1$, $c_2 \in (a, a + \Delta x)$ such that 
  \begin{equation*} 
    \partial^S_x u_y (c_1, d) \leq \frac{E(\Delta x, \Delta y)}{\Delta x \Delta y} \leq \partial^S_x u_y (c_2, d).
  \end{equation*}
  Since $\partial^S_x u_y $ is assumed to be continuous, we conclude that 
  \begin{equation*} 
    \lim_{(\Delta x, \Delta y) \to (0, 0)} \partial^S_x u_y(c_1, d) = \partial^S_x u_y(a, b) = \lim_{(\Delta x, \Delta y) \to (0, 0)} \partial^S_x u_y(c_2, d)
  \end{equation*}
  and hence 
  \begin{equation*} 
    \partial^S_x u_y(a, b) = \lim_{(\Delta x, \Delta y) \to (0, 0)} \frac{E(\Delta x, \Delta y)}{\Delta x \Delta y}
  \end{equation*}
  by the squeeze theorem again.

  Finally, since the limits for $\partial^S_x u_y(a, b)$ and $\partial^S_y u_x(a, b)$ are same, we have established the desired result.
\end{proof}

\subsection{The specularly normal vector}

In the classical derivative sense, for a differentiable function $u$ in $n$-dimensions, the vector $\left(u_{x_1}, u_{x_2}, \ldots, u_{x_n}, -1 \right)$ is perpendicular to the surface of $u$.
In this subsection, we extend this concept in the specular derivative sense.

\begin{definition}
  Let $\Omega$ be an open set in $\mathbb{R}^{n}$ and $a$ be a point in $\Omega$.
  Let $u$ be a weakly specularly differentiable function on $\Omega$.
  We say a vector $v$ in $\mathbb{R}^{n+1}$ is a (\emph{specularly}) \emph{normal vector}, or simply is (\emph{specularly}) \emph{normal}, to the surface of $u$ at $a$ if there exists a weak specular tangent hyperplane to the graph of $u$ at $(a, u[a])$ which is perpendicular to $v$.
\end{definition}

From now on, we only deal with a specularly normal vector in two-dimensions and denote $(x_1, x_2)$ or $(x, y)$ to be a typical point in $\mathbb{R}^{2}$.
Here, the following three questions can be risen.
First, is the existence of a specularly normal vector guaranteed?
Second, if so, is a specularly normal vector unique?
Third, what is the explicit formula for a specularly normal vector?
To examine the above questions, we find the following fundamental lemma.

\begin{lemma} \label{Lem : lines joining the components of P}
  Let $\Omega$ be an open set in $\mathbb{R}^{2}$ and $a = (a_1, a_2)$ be a point in $\Omega$.
  Let $u:\Omega \to \mathbb{R}$ be a specularly differentiable function on $\Omega$.
  Let $p_1$, $q_1$, $p_2$, and $q_2$ be all components of $\mathcal{P}(u, a)$, where $p_i$ and $q_i$ in the $x_ix_3$-plane such that $p_i \innerprd e_i > q_i \innerprd e_i$ for each $i = 1, 2$.
  For each $i = 1, 2$, let $\ell_i$ be the line joining $p_i$ and $q_i$.
  Then the symmetric forms of the lines $\ell_1$ and $\ell_2$ are given by 
  \begin{equation*}
    \ell_1 : \left( \frac{1}{\sqrt{1 + \alpha_1^2}} + \frac{1}{\sqrt{1 + \beta_1^2}} \right)^{-1}\left( x_1 + \frac{1}{\sqrt{1 + \beta_1^2}} \right) = \left( \frac{\alpha_1}{\sqrt{1 + \alpha_1^2}} + \frac{\beta_1}{\sqrt{1 + \beta_1^2}} \right)^{-1}\left( x_3 + \frac{\beta_1}{\sqrt{1 + \beta_1^2}} \right)
  \end{equation*}
  and 
  \begin{equation*}
    \ell_2 : \left( \frac{1}{\sqrt{1 + \alpha_2^2}} + \frac{1}{\sqrt{1 + \beta_2^2}} \right)^{-1}\left( x_2 + \frac{1}{\sqrt{1 + \beta_2^2}} \right) = \left( \frac{\alpha_2}{\sqrt{1 + \alpha_2^2}} + \frac{\beta_2}{\sqrt{1 + \beta_2^2}} \right)^{-1}\left( x_3 + \frac{\beta_2}{\sqrt{1 + \beta_2^2}} \right)
  \end{equation*}
  for $(x_1, x_2, x_3) \in \mathbb{R}^{3}$, where $\alpha_1 := \partial^R_{x_1} u(a)$, $\beta_1 := \partial^L_{x_1} u(a)$, $\alpha_2 := \partial^R_{x_2} u(a)$, and $\beta_2 := \partial^L_{x_2} u(a)$.
\end{lemma}

Before the proof we provide the following figure for Lemma \ref{Lem : lines joining the components of P}.
In Figure \ref{Fig : The lines l_1 and l_2}, we consider a simple case when $\overline{a} = (0, 0, 0)$.
The red rays and the blue rays illustrate $\pht{x_1}u$ and $\pht{x_2}u$, respectively.
Also, the purple lines are $\ell_1$ and $\ell_2$.

\begin{figure}[H] 
  \centering 
  \includegraphics[scale=0.9]{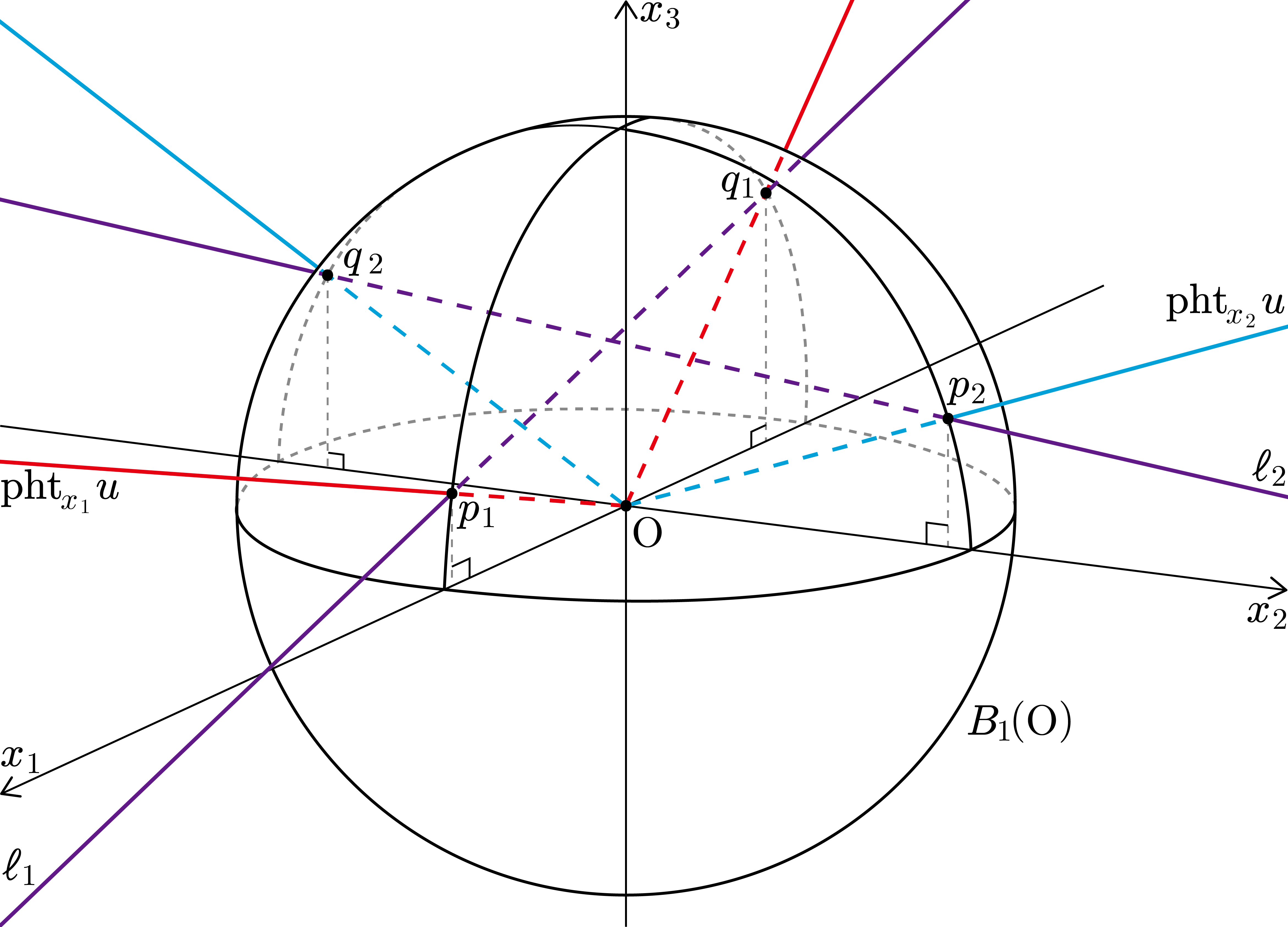} 
  \caption{The lines $\ell_1$ and $\ell_2$}
  \label{Fig : The lines l_1 and l_2}
\end{figure}

Now, we prove Lemma \ref{Lem : lines joining the components of P} as follows.

\begin{proof} 
  Consider an unit open ball $B_1 = B_1(\overline{a})$ centered at $\overline{a}$, where $\overline{a} = (a_1, a_2, u(a_1, a_2))$.
  Then $p_i$ and $q_i$ are the intersection points of the phototangent of $u$ at $a$ with respect to $x_i$ and a sphere $\partial B_1$ such that $p_i \innerprd e_i > q_i \innerprd e_i$ for each $i = 1, 2$.
  Using basic geometry properties, one can find that 
  \begin{equation*} 
    p_1 = \left( \frac{1}{\sqrt{1 + \alpha_1^2}}, 0, \frac{\alpha_1}{\sqrt{1 + \alpha_1^2}} \right) 
    \qquad \text{and} \qquad
    q_1 = \left( \frac{-1}{\sqrt{1 + \beta_1^2}}, 0, \frac{-\beta_1}{\sqrt{1 + \beta_1^2}} \right), 
  \end{equation*}
  \begin{equation*} 
    p_2 = \left( 0, \frac{1}{\sqrt{1 + \alpha_2^2}}, \frac{\alpha_2}{\sqrt{1 + \alpha_2^2}} \right) 
    \qquad \text{and} \qquad
    q_2 = \left( 0, \frac{-1}{\sqrt{1 + \beta_2^2}}, \frac{-\beta_2}{\sqrt{1 + \beta_2^2}} \right),
  \end{equation*}
  regardless of the signs of the semi-derivatives $\alpha_1$, $\beta_1$, $\alpha_2$, and $\beta_2$ (see Figure \ref{Fig : points p and q}).
  
  \begin{figure}[H] 
    \centering 
    \includegraphics[scale=1]{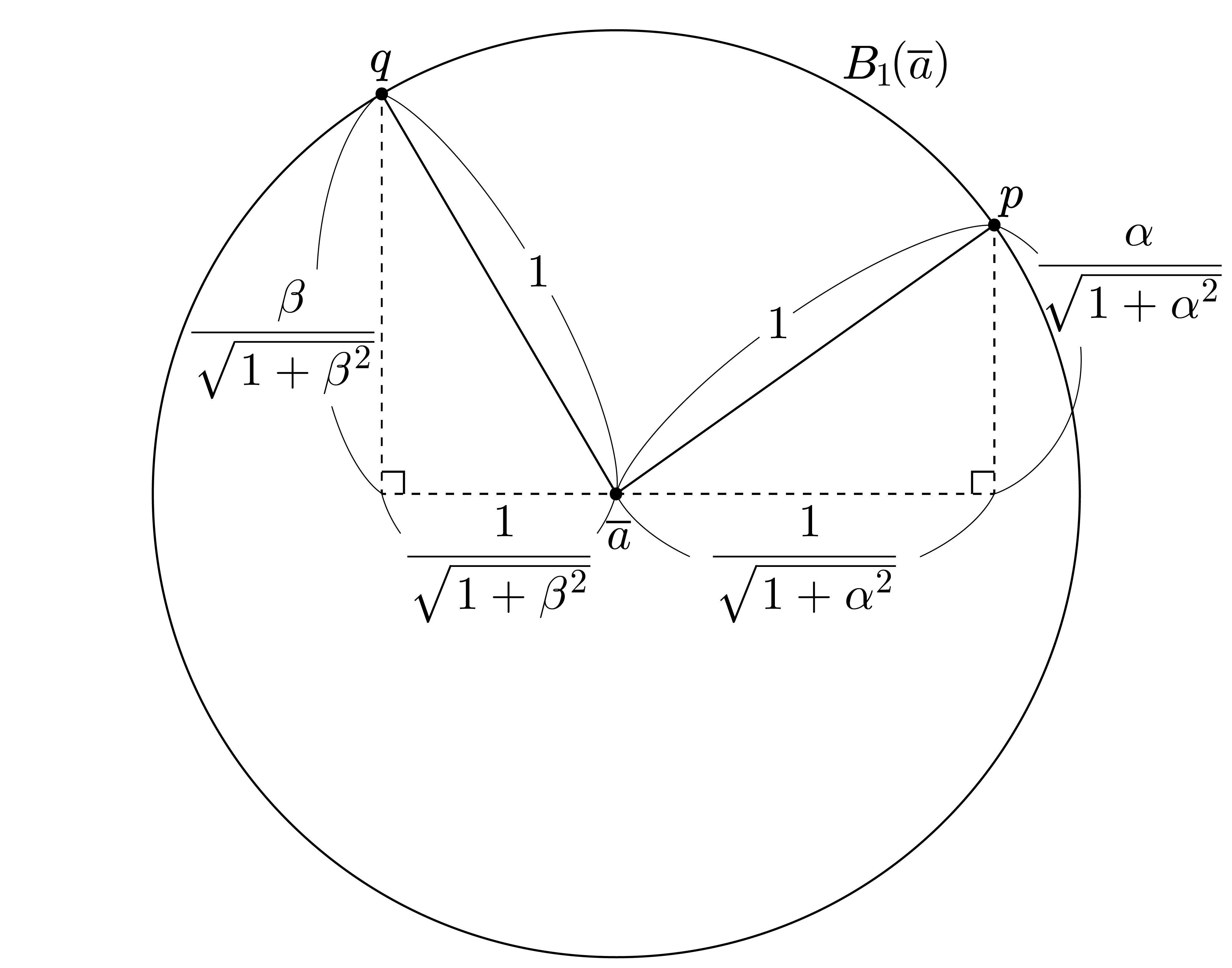} 
    \caption{Points $p$ and $q$}
    \label{Fig : points p and q}
  \end{figure}

  \noindent Hence, the symmetric forms of the lines $\ell_1$ and $\ell_2$ are given by 
  \begin{equation*}
    \ell_1 : \frac{x_1 - q_1 \innerprd e_1}{p_1 \innerprd e_1 - q_1 \innerprd e_1} = \frac{x_3 - q_1 \innerprd e_3}{p_1 \innerprd e_3 - q_1 \innerprd e_3}
  \end{equation*}
  and 
  \begin{equation*}
    \ell_2 : \frac{x_2 - q_2 \innerprd e_2}{p_2 \innerprd e_2 - q_2 \innerprd e_2} = \frac{x_3 - q_2 \innerprd e_3}{p_2 \innerprd e_3 - q_2 \innerprd e_3}
  \end{equation*}
  for $(x_1, x_2, x_3) \in \mathbb{R}^{3}$, completing the proof.
\end{proof}

Not every specularly differentiable function has a unique specular tangent. 
One can check whether a given function has a strong specular tangent hyperplane by using the following criterion.

\begin{theorem} \label{Thm : strong specular tangent hyperplane criterion}
  \emph{(The Strong Specular Tangent Hyperplane Criterion)} 
  Let $\Omega$ be an open set in $\mathbb{R}^{2}$ and $a = (a_1, a_2)$ be a point in $\Omega$.
  If a function $u:\Omega \to \mathbb{R}$ has a strong specular tangent hyperplane to the graph of $u$ at $(a, u[a])$, then 
  \begin{equation} \label{Eq : strong specular tangent hyperplane criterion}
    (\alpha_1 - \beta_1)\left( \sqrt{1 + \alpha_2^2} + \sqrt{1 + \beta_2^2}\right) - (\alpha_2 - \beta_2)\left( \sqrt{1 + \alpha_1^2} + \sqrt{1 + \beta_1^2}\right)= 0,
  \end{equation}
  where $\alpha_1 := \partial^R_{x_1} u(a)$, $\beta_1 := \partial^L_{x_1} u(a)$, $\alpha_2 := \partial^R_{x_2} u(a)$, and $\beta_2 := \partial^L_{x_2} u(a)$.
\end{theorem}

\begin{proof}
  Consider the lines $\ell_1$ and $\ell_2$ in Lemma \ref{Lem : lines joining the components of P}.
  Since the lines $\ell_1$ and $\ell_2$ are not parallel and lie on the same plane containing the four points $p_1$, $q_1$, $p_2$, and $q_2$, there exists have a unique intersection point of the lines $\ell_1$ and $\ell_2$.
  Hence, we obtain that 
  \begin{equation*}
    \det
    \begin{pmatrix}
      \displaystyle  \frac{1}{\sqrt{1 + \beta_1^2}} & \displaystyle \frac{-1}{\sqrt{1 + \beta_2^2}}  & \displaystyle \frac{\beta_1}{\sqrt{1 + \beta_1^2}} - \frac{\beta_2}{\sqrt{1 + \beta_2^2}} \\
      \displaystyle  \frac{1}{\sqrt{1 + \alpha_1^2}} + \frac{1}{\sqrt{1 + \beta_1^2}}  & 0 & \displaystyle \frac{\alpha_1}{\sqrt{1 + \alpha_1^2}} + \frac{\beta_1}{\sqrt{1 + \beta_1^2}}  \\
      \displaystyle  0 & \displaystyle \frac{1}{\sqrt{1 + \alpha_2^2}} + \frac{1}{\sqrt{1 + \beta_2^2}} & \displaystyle \frac{\alpha_2}{\sqrt{1 + \alpha_2^2}} + \frac{\beta_2}{\sqrt{1 + \beta_2^2}}
    \end{pmatrix}
    = 0,
  \end{equation*}
  which implies the criterion \eqref{Eq : strong specular tangent hyperplane criterion}; we omit the cumbersome calculation.
\end{proof}

If $u$ is classically differentiable, then $\alpha_1 = \beta_1$ and $\alpha_2 = \beta_2$ so that the criterion \eqref{Eq : strong specular tangent hyperplane criterion} always holds true.
Hence, Theorem \ref{Thm : strong specular tangent hyperplane criterion} makes sense in the classical derivative sense. 

\begin{example}
  Consider the function $u(x, y) = |x| - |y| - x - y$ for $(x, y) \in \mathbb{R}^{2}$, which has four weak specular tangent hyperplanes at $(0, 0, 0)$ (see \cite[Example 3.16]{2022_Jung}).
  Note that 
  \begin{equation*} 
    \partial^R_x u(0, 0) = 0 = \partial^L_y u(0, 0)
    \qquad \text{and} \qquad
    \partial^L_x u(0, 0) = -2 = \partial^R_y u(0, 0).
  \end{equation*}
  Indeed, the criterion \eqref{Eq : strong specular tangent hyperplane criterion} is not equal with zero:
  \begin{equation*} 
    2\left(\sqrt{5} + 1 \right) - (-2)\left( 1 + \sqrt{5} \right) = 4\left(1+ \sqrt{5} \right) \neq 0.
  \end{equation*}
  
\end{example}

Fortunately, the strong specular tangent hyperplane analogously preserves the classical derivatives' property that the vector perpendicular to the surface of a function and partial derivatives are related.

\begin{theorem} \label{Thm : specularly normal vector}
  Let $\Omega$ be an open set in $\mathbb{R}^{2}$ and $a$ be a point in $\Omega$.
  If $u$ is has a strong specular tangent hyperplane to the graph of $u$ at $(a, u[a])$, then the specularly normal vector of $u$ at $a$ is 
  \begin{equation*} 
    \left( \partial^S_{x_1} u(a), \partial^S_{x_2} u(a), -1 \right).
  \end{equation*}
\end{theorem}

\begin{proof} 
  Consider the lines $\ell_1$ and $\ell_2$ in Lemma \ref{Lem : lines joining the components of P}.
  Then the vector 
  \begin{equation} \label{Eq : the specular normal vector 1}
    \left( \frac{1}{\sqrt{1 + \alpha_1^2}} + \frac{1}{\sqrt{1 + \beta_1^2}}, 0, \frac{\alpha_1}{\sqrt{1 + \alpha_1^2}} + \frac{\beta_1}{\sqrt{1 + \beta_1^2}} \right) \times \left( 0, \frac{1}{\sqrt{1 + \alpha_2^2}} + \frac{1}{\sqrt{1 + \beta_2^2}}, \frac{\alpha_2}{\sqrt{1 + \alpha_2^2}} + \frac{\beta_2}{\sqrt{1 + \beta_2^2}} \right)
  \end{equation}
  is normal to the lines $\ell_1$ and $\ell_2$.
  Through the cumbersome calculation, one can find that the vector \eqref{Eq : the specular normal vector 1} is equal with 
  \begin{equation} \label{Eq : the specular normal vector 2}
    -K(A + B)(C + D)\left( \frac{\alpha_1 B + \beta_1 A}{A + B}, \frac{\alpha_2 D + \beta_2 C}{C + D}, -1 \right),
  \end{equation}
  where $A := \sqrt{1 + \alpha_1^2}$, $B := \sqrt{1 + \beta_1^2}$, $C := \sqrt{1 + \alpha_2^2}$, $D := \sqrt{1 + \beta_2^2}$, and $K := (ABCD)^{-1}$.
  Ignoring the coefficients $-K(A + B)(C + D)$ of the vector \eqref{Eq : the specular normal vector 2}, we have the vector 
  \begin{equation} \label{Eq : the specular normal vector 3}
    \left( \frac{\alpha_1 B + \beta_1 A}{A + B}, \frac{\alpha_2 D + \beta_2 C}{C + D}, -1 \right).
  \end{equation}
  
  Now, we claim that that \eqref{Eq : the specular normal vector 3} is equal with the vector $\left( \partial^S_{x_1} u(a), \partial^S_{x_2} u(a), -1 \right)$.
  To show this claim, we first show that 
  \begin{equation*} 
    \frac{\alpha_1 B + \beta_1 A}{A + B} = \partial^S_{x_1} u(a).
  \end{equation*}
  If $\alpha_1 + \beta_1 = 0$, then $A = B$ so that 
  \begin{equation*} 
    \frac{\alpha_1 A - \alpha_1 A}{A + A} - 0 = 0;
  \end{equation*}
  if $\alpha_1 + \beta_1 \neq 0$, then 
  \begin{equation*} 
    \frac{\alpha_1 B + \beta_1 A}{A + B} - \frac{\alpha_1 \beta_1 - 1 + AB}{\alpha_1 + \beta_1}  = \frac{-A + A^2 B - B + AB^2 - \alpha_1^2 B - \beta_1^2 A}{(\alpha_1 + \beta_1)(A + B)} = 0.
  \end{equation*}
  As the same way, for the variable $x_2$, one can show that 
  \begin{equation*} 
    \frac{\alpha_2 D + \beta_2 C}{C + D} = \partial^S_{x_2} u(a),
  \end{equation*}
  as required.
\end{proof}

If a function $u$ is weakly specularly differentiable on $\mathbb{R}^{2}$, it may not hold that there exists a weak specular tangent hyperplane which is perpendicular to the vector $\left( \partial^S_x u, \partial^S_y u, -1 \right)$.
In other words, the condition that a function $u$ has a strong specular tangent hyperplane at each point in $\mathbb{R}^{2}$ seems to be necessary in order to deal with the vector which is perpendicular to the surface of a function.
Here, we provide a counterexample as follows.

\begin{example} \label{Ex : strong stg is necessary}
  Consider the function $u : \mathbb{R}^{2} \to \mathbb{R}$ defined by 
  \begin{equation*} 
    u(x, y) = \frac{1}{2}(x + |x|) + \frac{1}{2}y + \frac{3}{2}|y|
  \end{equation*}
  for $(x, y) \in \mathbb{R}^{2}$.
  Then $u$ is weakly specularly differentiable at $(0, 0)$.
  In fact, one can calculate that $\partial^R_x u(0, 0) = 1$, $\partial^L_x u(0, 0) = 0$, $\partial^R_y u(0, 0) = 2$, $\partial^L_y u(0, 0) = -1$, and the vector
  \begin{equation} \label{Eq : strong stg is necessary - 1}
    \left( \partial^S_x u(0, 0), \partial^S_y u(0, 0), -1 \right)= \left( A(1, 0), A(2, -1), -1 \right) = \left( \sqrt{2} - 1, \sqrt{10} - 3, -1 \right).
  \end{equation}
  Then, the criterion \eqref{Eq : strong specular tangent hyperplane criterion} yields that 
  \begin{equation*} 
    1 \cdot \left(\sqrt{5} + \sqrt{2}\right) - 3\left(\sqrt{2} + 0\right) = \sqrt{5} - 2\sqrt{2} \neq 0,
  \end{equation*}
  which implies that $u$ has at least two weak specular tangent hyperplane at $(0, 0, 0)$.
  Now, we find whether $u$ has a weak specular tangent hyperplane which is perpendicular to the vector \eqref{Eq : strong stg is necessary - 1}.
  Note that 
  \begin{equation*} 
    \mathcal{P}(u, (0, 0)) = \left\{ \left( \frac{1}{\sqrt{2}}, 0, \frac{1}{\sqrt{2}}\right), (-1, 0, 0), \left( 0, \frac{1}{\sqrt{5}}, \frac{2}{\sqrt{5}}\right), \left( 0, -\frac{1}{\sqrt{2}}, \frac{1}{\sqrt{2}}\right) \right\}
  \end{equation*}
  and the four weak specular tangent hyperplanes are as follows:
  \begin{align*}
    \operatorname{wstg}_1 u :& \left( \sqrt{2} - 1 \right)x - \left( \sqrt{10} - \sqrt{5} - 2 \right)y + \sqrt{2} - 1 = z, \\
    \operatorname{wstg}_2 u :& \left( \sqrt{2} - 1 \right)x - \left( \sqrt{2} - 1 \right)y + \sqrt{2} - 1 = z, \\
    \operatorname{wstg}_3 u :& -\left( \sqrt{10} - 3 \right)x + \left( \sqrt{10} - 3 \right)y + \sqrt{5} - \sqrt{2} = z,  \\
    \operatorname{wstg}_4 u :& \left( \sqrt{5} - \sqrt{2} \right)x + \left( \sqrt{10} - 3 \right)y + \sqrt{5} - \sqrt{2} = z, 
  \end{align*}
  for $(x, y, z) \in \mathbb{R}^{3}$.
  Clearly, there does not exist a weak weak specular tangent hyperplane which is perpendicular to the vector \eqref{Eq : strong stg is necessary - 1}.
\end{example}

To answer the previous three questions, given a point in $\mathbb{R}^{3}$, a function $u$ having a strong specular tangent hyperplane at each point in $\mathbb{R}^{2}$ has a unique specularly normal vector which is specularly normal to the surface of $u$ and can be expressed as specularly partial derivatives of $u$.

\section{One-dimensional differential equations with specular derivatives}

In this section, we can construct the wave equation with specular derivatives in one-dimension. 
We deal with only infinite domain and start from homogeneous problems to nonhomogeneous problems.

\subsection{Transport equation with specular derivatives in infinite domain}

In order to solve a wave equation with specular derivatives, it is necessary to address a transport equation with specular derivatives.
Consider the homogeneous transport equation with specular derivatives on $\mathbb{R}$:
\begin{equation} \label{PDE: homo transport eq. w/ spd in 1-dim}
  \partial^S_t u + \partial^S_x u = 0, \quad (x, t) \in \mathbb{R} \times (0, \infty).  
\end{equation}

\begin{theorem} \label{Thm: homo transport eq. w/ spd in 1-dim}
  Assume that a function $u:\mathbb{R} \times (0, \infty) \to \mathbb{R}$ has a strong specular tangent hyperplane at each point $(x, t)\in \mathbb{R} \times (0, \infty)$.
  The general solution of \eqref{PDE: homo transport eq. w/ spd in 1-dim} is given by 
  \begin{equation} \label{Sol: homo transport eq. w/ spd in 1-dim}
    u(x, t) = h(x - t),
  \end{equation}
  where $h$ is an arbitrary specularly differentiable function of a single-variable.
\end{theorem}

\begin{proof} 
  Let $u=u(x, t)$ be a solution of \eqref{PDE: homo transport eq. w/ spd in 1-dim}.
  Then the solution surface $F$ can be written by 
  \begin{equation*} 
    F(x, t, z) = u(x, t) - z = 0  
  \end{equation*}
  and the normal vector to the solution surface $F$ is 
  \begin{equation*} 
    D^S F(x, t, z) = \left( \partial^S_x u, \partial^S_t u, -1  \right)
  \end{equation*}
  due to Theorem \ref{Thm : specularly normal vector}.
  For the vector of coefficients of \eqref{PDE: homo transport eq. w/ spd in 1-dim}, one can calculate that 
  \begin{equation*} 
    (1, 1, 0) \innerprd D^S F(x, t, z) = \partial^S_x u + \partial^S_t u = 0,
  \end{equation*}
  which yields that $(1, 1, 0)$ is perpendicular to $D^S F(x, t, z)$.
  Since $D^S F(x, t, z)$ is normal to the surface $F$, $(1, 1, 0)$ is tangent to the surface $F$.

  Now, consider a parametrized curve
  \begin{equation*} 
    \gamma(s) = (x(s), t(s), z(s))
  \end{equation*}
  for $s$ in some open interval.
  Then the tangent to the curve is 
  \begin{equation*} 
    (1, 1, 0) = \frac{d\gamma}{d^S s}  = \left( \frac{dx}{d^S s}, \frac{dt}{d^S s}, \frac{dz}{d^S s} \right)
  \end{equation*}
  as the vector $(1, 1, 0)$ is tangent to the surface $F$.
  Hence, the parametric form of the characteristic equations is given by
  \begin{equation} \label{Syt: homo transport eq. w/ spd in 1-dim}
  \begin{cases} 
    \displaystyle \frac{dx}{d^S s} = 1  ,\\[0.25cm] 
    \displaystyle \frac{dt}{d^S s} = 1  ,\\[0.25cm] 
    \displaystyle \frac{du}{d^S s} = 0  ,
  \end{cases}
  \end{equation}
  which is ODE system with specular derivatives. 
  Since the right-hand sides of all equations in \eqref{Syt: homo transport eq. w/ spd in 1-dim} are constant, the system is ODE system with classical derivatives. 
  Hence, the solution of the ODE system \eqref{Syt: homo transport eq. w/ spd in 1-dim} is 
  \begin{equation*}
  \begin{cases} 
   x(s) = s + c_1 ,\\ 
   t(s) = s + c_2 ,\\ 
   u(s) = c_3, 
  \end{cases}
  \end{equation*}
  for some constants $c_1$, $c_2$, and $c_3$.
  Observe that 
  \begin{equation*} 
    x(s) - t(s) = c_1 - c_2 =: c_4
  \end{equation*}
  for some constant $c_4$.
  Since $u(x, t)$ is constant along the lines $x - t=c_4$, we obtain the general solution of \eqref{PDE: homo transport eq. w/ spd in 1-dim}  
  \begin{equation*} 
    u(x, t) = h(x - t),
  \end{equation*}
  for an arbitrary function $h$.
\end{proof}

\subsection{Wave equation with specular derivatives in infinite domain}

Consider the homogeneous wave equation with specular derivatives on $\mathbb{R}$:
\begin{equation} \label{PDE: homo wave eq. w/ spd in 1-dim}
  \partial^S_t u_t - \partial^S_x u_x = 0, \quad (x, t) \in \mathbb{R} \times (0, \infty).  
\end{equation}
To solve this equation, we introduce a constrained hypothesis as follows:
\begin{enumerate}[label=(H)] 
\rm\item\label{H} $(\partial_t - \partial_x)u$ has a strong specular tangent hyperplane at each point $(x, t) \in \mathbb{R} \times (0, \infty)$.  \label{(H)}
\end{enumerate}
Eventually, we will drop the hypothesis \ref{(H)} to solve the nonhomogeneous wave equation with specular derivatives on $\mathbb{R}$.

\begin{theorem}
  Assume the hypothesis \emph{\ref{(H)}}.
  The general solution of \eqref{PDE: homo wave eq. w/ spd in 1-dim} is given by 
  \begin{equation} \label{Sol: homo wave eq. w/ spd in 1-dim}
    u(x, t) = g(x + t) + h(x - t),
  \end{equation}
  where $g$ and $h$ are arbitrary twice specularly differentiable functions of a single-variable.
\end{theorem}

\begin{proof} 
  Thanks to Theorem \ref{Thm : Symmetry of second specular derivatives in 2d}, observing that 
  \begin{equation*} 
    \left( \partial^S_t + \partial^S_x \right) \left( \partial_t - \partial_x \right) u 
    = \left( \partial^S_t \partial_t + \partial^S_x \partial_t - \partial^S_t \partial_x - \partial^S_x \partial_x \right) u 
    = \left( \partial^S_t \partial_t + \partial^S_t \partial_x - \partial^S_t \partial_x - \partial^S_x \partial_x \right) u 
    = \left( \partial^S_t \partial_t - \partial^S_x \partial_x \right) u,
  \end{equation*}
  the equation can be rewritten as 
  \begin{equation*} 
    \left( \partial^S_t + \partial^S_x \right) \left( \partial_t - \partial_x \right) u = 0.
  \end{equation*}
  Writing 
  \begin{equation*} 
      v \equiv \left( \partial_t - \partial_x \right)u,
  \end{equation*}
  we have 
  \begin{equation*} 
    \left( \partial^S_t + \partial^S_x \right) v = 0,
  \end{equation*}
  which means that $v$ solves a first order transport equation with specular derivatives.
  By Theorem \ref{Thm: homo transport eq. w/ spd in 1-dim}, the general solution for $v$ is given by 
  \begin{equation*} 
    v = h(x - t)
  \end{equation*}
  for an arbitrary specularly differentiable function $h$ of a single variable.

  Now, it remains to solve 
  \begin{equation*} 
    u_t - u_x = h(x - t),
  \end{equation*}
  which is the transport equation with classical derivatives. 
  Hence, we conclude that \eqref{Sol: homo wave eq. w/ spd in 1-dim} is the general solution of \eqref{PDE: homo wave eq. w/ spd in 1-dim}.
\end{proof}

Next, consider the initial value problem for the homogeneous wave equation with specular derivatives on $\mathbb{R}$:
\begin{equation} \label{PDE: IVP homo wave eq. w/ spd in 1-dim}
\begin{cases} 
  \displaystyle
  \partial^S_t u_t - \partial^S_x u_x = 0, & (x, t) \in \mathbb{R} \times (0, \infty),\\ 
  u(x, 0) = \varphi(x), & x \in \mathbb{R},\\ 
  u_t(x, 0) = \psi(x), & x \in \mathbb{R} , 
\end{cases}
\end{equation}
where $\varphi \in S^{2}(\mathbb{R})$ and $\psi \in S^1(\mathbb{R})$ are given.
We seek a formula for $u \in S^2(\mathbb{R} \times [0, \infty))$ solving \eqref{PDE: IVP homo wave eq. w/ spd in 1-dim} in terms of $\varphi$ and $\psi$.
A famous formula, so-called \emph{d'Alembert's formula}, still works in the wave equation with specular derivatives.

\begin{theorem} \label{Thm: d'Alembert's formula}
  Assume the hypothesis \emph{\ref{(H)}}.
  The $S^2$-solution of \eqref{PDE: IVP homo wave eq. w/ spd in 1-dim} is given by 
  \begin{equation} \label{Sol : d'Alembert's formula}
    u(x, t) = \frac{1}{2} [\varphi(x + t) + \varphi(x - t)] + \frac{1}{2} \int_{x - t}^{x + t} \psi(s) ~ds,
  \end{equation}
  which is d'Alembert's formula.
\end{theorem}

\begin{proof} 
  Let $u(x, t)$ be a $S^2$-solution which is given by 
  \begin{equation} \label{Thm: d'Alembert's formula - 1}
      u(x, t) = g(x + t) + h(x - t)
  \end{equation}
  for some functions $g$ and $h$.
  Evaluating $u$ at $t=0$ yields that 
  \begin{equation} \label{Thm: d'Alembert's formula - 2}
    u(x, 0) = g(x) + h(x) = \varphi(x)
  \end{equation}
  and differentiating this equation with respect to $x$ implies that 
  \begin{equation} \label{Thm: d'Alembert's formula - 3}
    \varphi'(x) = g'(x) + h'(x).
  \end{equation}
  Evaluating $u_t$ at $t=0$, we have 
  \begin{equation} \label{Thm: d'Alembert's formula - 4}
    u_t(x, 0) = g'(x) + h'(x) = \psi(x).
  \end{equation}
  The solution of the system of \eqref{Thm: d'Alembert's formula - 3} and \eqref{Thm: d'Alembert's formula - 4} is 
  \begin{equation*} 
    g'(x) = \frac{1}{2}\varphi'(x) + \frac{1}{2}\psi(x)
  \end{equation*}
  and integration of this equation yields that 
  \begin{equation} \label{Thm: d'Alembert's formula - 5}
    g(x) = \frac{1}{2}\varphi(x) + \frac{1}{2}\int_{0}^{x}\psi(s)~ds + c
  \end{equation}
  for some constant $c \in \mathbb{R}$.
  Combining \eqref{Thm: d'Alembert's formula - 2} and \eqref{Thm: d'Alembert's formula - 5}, we find that 
  \begin{equation} \label{Thm: d'Alembert's formula - 6}
    h(x) = \varphi(x) - g(x) =\frac{1}{2}\varphi(x) - \frac{1}{2}\int_{0}^{x}\psi(s)~ds - c.
  \end{equation}
  Finally, substituting \eqref{Thm: d'Alembert's formula - 5} and \eqref{Thm: d'Alembert's formula - 6} into \eqref{Thm: d'Alembert's formula - 1}, we obtain d'Alembert's formula \eqref{Sol : d'Alembert's formula}.
\end{proof}

Here, the conventional application of d'Alembert's formula is possible.

\begin{remark}
  Consider the initial and boundary value problem for the homogeneous wave equation with specular derivatives on the half-line $\mathbb{R}_+=\left\{ x>0 \right\}$:  
  \begin{equation} \label{PDE: IVP, BVP homo wave eq. w/ spd in 1-dim}
    \begin{cases} 
      \displaystyle
      \partial^S_t u_t - \partial^S_x u_x = 0, & (x, t) \in \mathbb{R}_+ \times (0, \infty),\\ 
      u(x, 0) = \varphi(x), & x \in \mathbb{R}_+,\\ 
      u_t(x, 0) = \psi(x), & x \in \mathbb{R}_+ , \\
      u(0, t) = 0, & t \in (0, \infty),
    \end{cases}
  \end{equation}
  where $\varphi \in S^{2}(\mathbb{R})$ and $\psi \in S^1(\mathbb{R})$ are given, with $\varphi(0) = \psi(0) = 0$.
  Assume the hypothesis \ref{(H)}.
  The reflection method implies that for $x \geq 0$ and $t \geq 0$
  \begin{equation} \label{Sol : IVP, BVP homo wave eq. w/ spd in 1-dim}
    u(x, t) =
    \begin{cases} 
      \displaystyle
      \frac{1}{2} [\varphi(x + t) + \varphi(x - t)] + \frac{1}{2} \int_{x - t}^{x + t} \psi(s) ~ds & \text{if } x \geq t \geq 0,\\[0.4cm] 
      \displaystyle
      \frac{1}{2} [\varphi(x + t) - \varphi(t - x)] + \frac{1}{2} \int_{t - x}^{x + t} \psi(s) ~ds & \text{if } t \geq x \geq 0, 
    \end{cases}    
  \end{equation}
  is a solution of \eqref{PDE: IVP, BVP homo wave eq. w/ spd in 1-dim}.
\end{remark}

Here, we provide an example of the initial and boundary value problem \eqref{PDE: IVP, BVP homo wave eq. w/ spd in 1-dim} and its $S^2$-solution.

\begin{example}
  Consider 
  \begin{equation} \label{PDE: IVP homo wave eq. w/ spd in 1-dim example}
    \begin{cases} 
      \displaystyle
      \partial^S_t u_t - \partial^S_x u_x = 0, & (x, t) \in (0, \infty) \times (0, \infty),\\[0.1cm]  
      \displaystyle
      u(x, 0) = q(x-1) + \frac{1}{2}x^2 + \frac{1}{2}, & x \in [0, \infty),\\[0.1cm] 
      u_t(x, 0) = |x - 1| - 1, & x \in [0, \infty) , \\
      u(0, t) = 0, & t \in (0, \infty),
    \end{cases}
    \end{equation}
    where $q$ is defined as in \eqref{Fnc : q}.
    Note that this problem satisfies the condition:
    \begin{equation*} 
      \varphi(0) = \psi(0) = \varphi^{\spd \spd}(0) = 0,
    \end{equation*}
    where $\varphi(x) := u(x, 0)$ and $\psi(x) := u_t(x, 0)$ for $x \in [0, \infty)$.

    The solution of \eqref{PDE: IVP homo wave eq. w/ spd in 1-dim example} is 
    \begin{equation*}
      u(x, t) =
      \begin{cases} 
      \displaystyle q(x + t - 1) + \frac{1}{2}x^2 + \frac{1}{2}t^2 - t + \frac{1}{2} & \text{if } x \geq t \geq 0,\\[0.1cm]
      \displaystyle q(t + x - 1) - q(t - x - 1) + xt - x & \text{if } t \geq x \geq 0,
      \end{cases}
    \end{equation*}
    thanks to d'Alembert's formula \eqref{Sol : IVP, BVP homo wave eq. w/ spd in 1-dim}.
    Clearly, this solution meets the boundary condition, i.e., $u = 0$ if $x = 0$. 
    From now on, we check whether this solution satisfies \eqref{PDE: IVP homo wave eq. w/ spd in 1-dim example}. 
    Figure \ref{Fig : Regions of interest separated by three lines} illustrates regions of interest separated by three lines: $t-x=1$, $t-x=0$, and $t+x=1$.

    For $x \geq t \geq 0$, we find that 
    \begin{equation*}
    u_x(x, t) 
    = |x + t - 1| + x
    =
    \begin{cases} 
    2x + t - 1 & \text{if } x + t \geq 1,\\ 
    -t + 1 & \text{if } x + t < 1
    \end{cases}
    \end{equation*}
    and 
    \begin{equation*}
      u_t(x, t) 
      = |x + t - 1| + t - 1
      =
      \begin{cases} 
      x + 2t - 2 & \text{if } x + t \geq 1,\\ 
      -x & \text{if } x + t < 1.
      \end{cases}
      \end{equation*}
    Then 
    \begin{align*}
      \partial^S_x u_x(x, t) 
      &=
      \begin{cases} 
      2 & \text{if } x + t > 1,\\
      A(2, 0) & \text{if } x + t = 1,\\
      0 & \text{if } x + t < 1
      \end{cases} \\
      &= \partial^S_t u_t(x, t)  .
    \end{align*}
    Next, for $t \geq x \geq 0$, observe that 
    \begin{equation*} 
      \left\{ (x, t) \in \mathbb{R}^{2} : t + x < 1 \leq t - x \text{ and } t \geq x \geq 0 \right\} = \varnothing;
    \end{equation*}
    we do not have to consider the case $t + x < 1$ and $t - x \geq 1$.
    Hence, we have 
    \begin{equation*}
      u_x(x, t) 
      = |t + x - 1| + |t - x - 1| + t - 1
      =
      \begin{cases} 
      3t - 3 & \text{if } t + x \geq 1, t - x \geq 1,\\ 
      2x + t - 1 & \text{if } x + t \geq 1, t - x < 1,\\ 
      -t + 1 & \text{if } x + t < 1, t - x < 1,
      \end{cases}
    \end{equation*}
    and 
    \begin{equation*}
      u_t(x, t) 
      = |t + x - 1| - |t - x - 1| + x
      =
      \begin{cases} 
      3x & \text{if } t + x \geq 1, t - x \geq 1,\\ 
      x + 2t - 2 & \text{if } x + t \geq 1, t - x < 1,\\ 
      -x & \text{if } x + t < 1, t - x < 1.
      \end{cases}
    \end{equation*}
    Then 
    \begin{align*}
      \partial^S_x u_x(x, t) 
      &=
      \begin{cases} 
      0 & \text{if } t + x > 1, t - x > 1,\\
      A(2, 0) & \text{if } t - x = 1,\\
      2 & \text{if } x + t > 1, t - x < 1,\\
      A(2, 0) & \text{if } t + x = 1,\\
      0 & \text{if } x + t < 1, t - x < 1 
      \end{cases} \\
      &= \partial^S_t u_t(x, t).
    \end{align*}
    Lastly, one can check that $u$, $u_t$, $u_x$ are continuous along $\left\{ t = x \right\}$ and that $\partial^S_x u_x$ and $\partial^S_t u_t$ satisfy Theorem \ref{Thm : strong specular tangent hyperplane criterion}.

    \begin{figure}[H] 
    \centering 
    \includegraphics[width=0.35\textwidth]{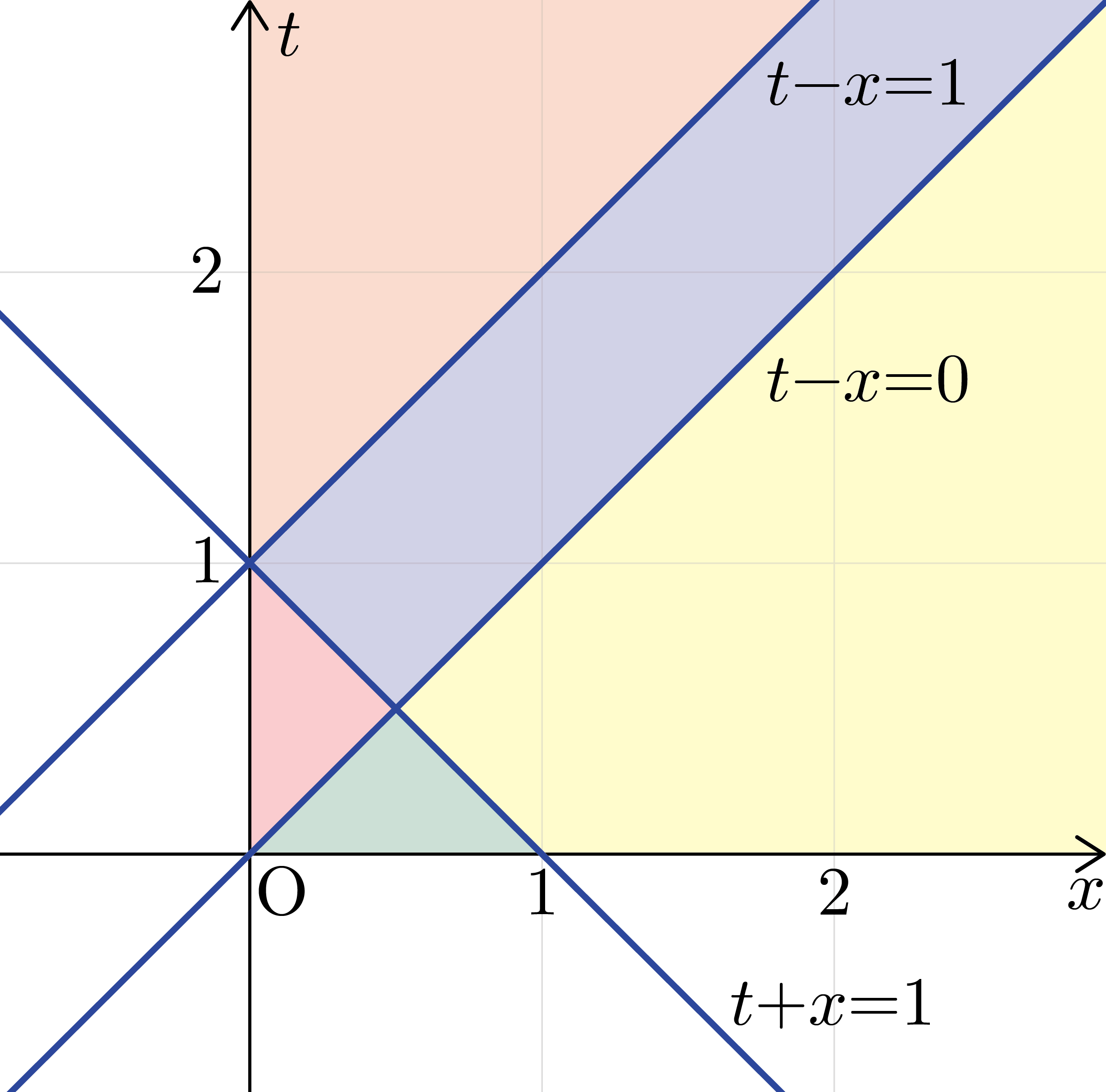} 
    \caption{Regions of interest separated by three lines}
    \label{Fig : Regions of interest separated by three lines}
    \end{figure}
    
\end{example}

From now on, consider the initial value problem for the nonhomogeneous wave equation with specular derivatives on $\mathbb{R}$:
\begin{equation} \label{PDE: IVP nonhomo wave eq. w/ spd in 1-dim}
  \begin{cases} 
    \displaystyle
    \partial^S_t u_t - \partial^S_x u_x = f(x, t), & (x, t) \in \mathbb{R} \times (0, \infty),\\ 
    u(x, 0) = \varphi(x), & x \in \mathbb{R},\\ 
    u_t(x, 0) = \psi(x), & x \in \mathbb{R} ,
  \end{cases}
  \end{equation}
where $\varphi \in S^{2}(\mathbb{R})$, $\psi \in S^1(\mathbb{R})$, and $f \in S^0(\mathbb{R}\times [0, \infty))$ are given.
We seek a solution $u \in S^2(\mathbb{R}\times [0, \infty))$ of \eqref{PDE: IVP nonhomo wave eq. w/ spd in 1-dim}.
In the classical derivative sense, there are three methods solving \eqref{PDE: IVP nonhomo wave eq. w/ spd in 1-dim}: reduction to first-order equations, using Green's theorem, and Duhamel's principle.
Among these methods, we try to solve \eqref{PDE: IVP nonhomo wave eq. w/ spd in 1-dim} by extending Green's theorem to the specular space.

The proof of Green's theorem with specular derivatives is inspired by that of Green's theorem with classical derivatives in \cite{2012_Colley_BOOK}.
It is needs to recall the following basic definitions (see \cite[Chapter 5]{2012_Colley_BOOK}):

\begin{definition}
  Let $E$ be a subset of $\mathbb{R}^{2}$.
  We say $E$ is an \emph{elementary region} in the plane if it can be described as one of the following three types:
  \begin{enumerate}[label=(\roman*)] 
  \rm\item \emph{type} \RN{1}: $E =\left\{ (x, y) : \omega_1(x) \leq y \leq \omega_2(x), a \leq x \leq b \right\}$,
  where $\omega_1$ and $\omega_2$ are continuous on $[a, b]$.
  \rm\item \emph{type} \RN{2}: $E =\left\{ (x, y) : \omega_3(y) \leq x \leq \omega_4(y), c \leq y \leq d \right\}$,
  where $\omega_3$ and $\omega_4$ are continuous on $[c, d]$.
  \rm\item \emph{type} \RN{3}: $E$ is of both type \RN{2} and type \RN{3}.
  \end{enumerate}  
\end{definition}

Here, we state the restricted Green's theorem with specular derivatives since it is enough to solve \eqref{PDE: IVP nonhomo wave eq. w/ spd in 1-dim}; to achieve literally generalization, the case that a domain $\overline{\Omega}$ in Theorem \ref{Thm : Green's Theorem with specular derivatives} is not an elementary region of type \RN{3} should be considered. 

\begin{theorem} \label{Thm : Green's Theorem with specular derivatives}
  \emph{(Green's Theorem with specular derivatives)} 
  Let $\Omega$ be an open, bounded subset of $\mathbb{R}^{2}$ whose its boundary $\partial \Omega$ consists of finitely many simple, closed, piecewise $C^1$ curves.
  Assume  $\overline{\Omega}$ is an elementary region of type \RN{3}.
  Let $F(x, y) = P(x, y) e_1 + Q(x, y) e_2$  be a vector field of class $S^1$ throughout $\overline{\Omega}$. 
  Then 
  \begin{equation*} 
    \iint_{\overline{\Omega}}\left(\frac{\partial P}{\partial^S x}-\frac{\partial Q}{\partial^S y}\right) ~dx dy = \oint_{\partial \Omega} P ~dy + Q ~dx,
  \end{equation*}
  where the boundary $\partial \Omega$ traverses in the counterclockwise direction.
\end{theorem}

\begin{proof} 
  Recalling that a type \RN{3} is one that of both type \RN{1} and type \RN{2} elementary regions, $\overline{\Omega}$ can be described in two way:
  \begin{align*}
    \overline{\Omega} 
    &= \left\{ (x, y) \in \mathbb{R}^{2} : \omega_1(x) \leq y \leq \omega_2(x), a \leq x \leq b \right\} \\
    &= \left\{ (x, y) \in \mathbb{R}^{2} : \omega_3(y) \leq x \leq \omega_4(y), c \leq y \leq d \right\},
  \end{align*}
  where $\omega_i$ is continuous and piecewise $C^1$ for each $i = 1, 2, 3, 4$ (see Figure \ref{Fig : Two viewpoints on Omega a type 3}).
  Note that the first description and the second description of $\overline{\Omega}$ are is as a type \RN{1} and a type \RN{2} elementary region, respectively.

  \begin{figure}[H] 
  \centering 
  \includegraphics[width=0.8\textwidth]{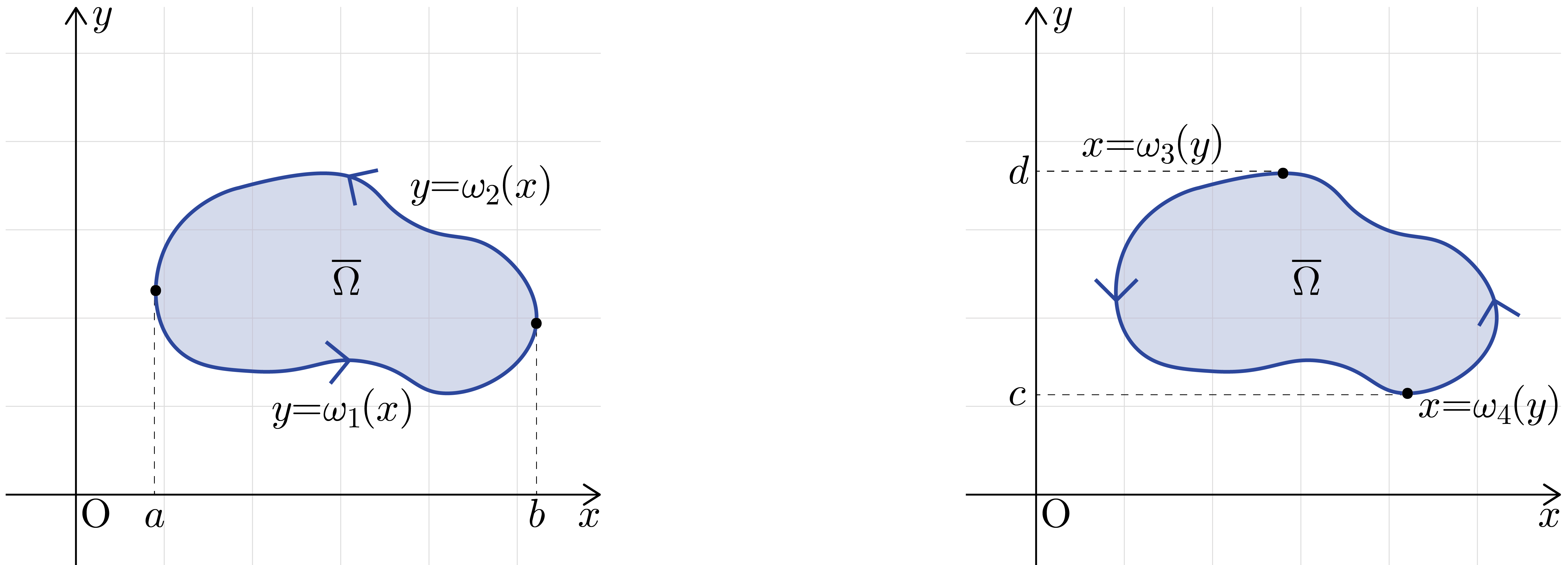} 
  \caption{Two viewpoints on $\overline{\Omega}$ a type \RN{3}}
  \label{Fig : Two viewpoints on Omega a type 3}
  \end{figure}
  
  First, we view the case $\overline{\Omega}$ as a type \RN{1}.
  Then $\partial \Omega$ consists of a lower curve $\Gamma_1$ and an upper curve $\Gamma_2$, which can be parametrized as follows:
  \begin{equation*} 
    \Gamma_1 : 
    \begin{cases} 
     x = t,\\ 
     y = \omega_1(t), 
    \end{cases}
    \qquad \text{and} \qquad
    \Gamma_2 :
    \begin{cases} 
      x = t,\\ 
      y = \omega_2(t),
    \end{cases}  
  \end{equation*}
  for $a \leq t \leq b$, where $\Gamma_1$ is oriented counterclockwise and $\Gamma_1$ is oriented clockwise.
  Now, we claim that 
  \begin{equation} \label{Thm : Green's Theorem with specular derivatives - 1}
    \iint_{\overline{\Omega}}- \frac{\partial Q}{\partial^S y}~dA = \oint_{\partial \Omega} Q~dx
  \end{equation}
  and prove it by evaluating each sides. 
  On the one hand, one can find that 
  \begin{align*}
    \iint_{\overline{\Omega}}- \frac{\partial Q}{\partial^S y}(x, y) ~dA 
    &= \int_a^b \int_{\omega_1(x)}^{\omega_2(x)} -\frac{\partial Q}{\partial^S y}(x, y) ~dy dx \\
    &= \int_a^b \left[ Q(x, \omega_1(x)) - Q(x, \omega_2(x)) \right] ~dx
  \end{align*}
  by Theorem \ref{Thm : the first form of FTC with spd}.
  On the other hand, we compute 
  \begin{align*}
    \oint_{\partial \Omega} Q(x, y)~dx 
    &= \int_{\Gamma_1} Q(x, y)~dx - \int_{\Gamma_2} Q(x, y)~dx \\
    &= \int_a^b Q(t, \omega_1(t)) ~dt - \int_a^b Q(t, \omega_2(t))~dt \\
    &= \int_a^b \left[Q(t, \omega_1(t)) - Q(t, \omega_2(t)) \right] ~dt.
  \end{align*}
  Combining these computations, the claim \eqref{Thm : Green's Theorem with specular derivatives - 1} is proved, in the case $\Omega$ is a type \RN{1}. 
  
  Second, if we view $\overline{\Omega}$ as a type \RN{2}, one can prove that 
  \begin{equation} \label{Thm : Green's Theorem with specular derivatives - 2}
    \iint_{\overline{\Omega}} \frac{\partial P}{\partial^S x} ~dA = \oint_{\partial \Omega} P ~dy,
  \end{equation}
  by proving both line integral and the double integral is equal with 
  \begin{equation*} 
    \int_c^d \left[ P(\omega_3(y), y) - P(\omega_4(y), y) \right]~dy
  \end{equation*}
  in an analogous way.

  Since $\overline{\Omega}$ is simultaneously of type \RN{1} and type \RN{2}, adding \eqref{Thm : Green's Theorem with specular derivatives - 1} and \eqref{Thm : Green's Theorem with specular derivatives - 2} yields that 
  \begin{align*}
    \oint_{\partial \Omega} P(x, y)~dx + Q(x, y)~dy &= \oint_{\partial \Omega} P~dx + \oint_{\partial \Omega} Q~dy \\
    &= \iint_{\overline{\Omega}} \frac{\partial P}{\partial^S x} ~dA - \iint_{\overline{\Omega}} \frac{\partial Q}{\partial^S y}~dA \\
    &= \iint_{\overline{\Omega}} \left( \frac{\partial P}{\partial^S x} - \frac{\partial Q}{\partial^S y}\right) ~dA,
  \end{align*}
  as required.
\end{proof}

As one can see in the proof, the condition that a given force is proper is required to use FTC.
Now, we can find the $S^2$-solution of \eqref{PDE: IVP nonhomo wave eq. w/ spd in 1-dim} by using the Green's theorem with specular derivatives.

\begin{theorem} \label{Thm: IVP nonhomo wave eq. w/ spd in 1-dim}
  The $S^2$-solution of \eqref{PDE: IVP nonhomo wave eq. w/ spd in 1-dim} is given by 
  \begin{equation} \label{Sol : IVP nonhomo wave eq. w/ spd}
    u(x, t) = \frac{1}{2} [\varphi(x + t) + \varphi(x - t)] + \frac{1}{2} \int_{x-t}^{x+t} \psi(s) ~ds + \frac{1}{2} \int_{0}^{t} \int_{x-t+s}^{x+t-s} f(y, s) ~dy ds.
  \end{equation}
\end{theorem}

\begin{proof} 
  Fix a point $(x_0, t_0)$ in $\Omega$.
  Set $\bigtriangleup := \left\{ (x, t) : 0 \leq t \leq t_0, |x - x_0| \leq |t - t_0| \right\}$, the domain of dependence of the point $(x_0, t_0)$.
  Integrating the wave equation $\partial^S_t u_t - \partial^S_x u_x = f(x, t)$ over $\bigtriangleup$, we obtain that 
  \begin{equation} \label{Thm: IVP nonhomo wave eq. w/ spd in 1-dim - 1}
    -\iint_{\bigtriangleup} \left( \frac{\partial u_x}{\partial^S x} - \frac{\partial u_t}{\partial^S t} \right) ~dx dt = \iint_{\bigtriangleup} f(x, t) ~dx dt.
  \end{equation}

  Next, assume the boundary $\partial \bigtriangleup$ of $\bigtriangleup$ traverses in the counterclockwise direction and split $\partial \bigtriangleup$ into three straight line segments as follows:
  \begin{align*}
    \Gamma_0 &:= \left\{ (x, 0) : x_0 - t_0 \leq x \leq x_0 + t_0 \right\}, \\
    \Gamma_1 &:= \left\{ (x, t) : x + t = x_0 + t_0\right\}, \\
    \Gamma_2 &:= \left\{ (x, t) : x - t = x_0 - t_0 \right\} .
  \end{align*}

  \begin{figure}[H] 
  \centering 
  \includegraphics[scale=1]{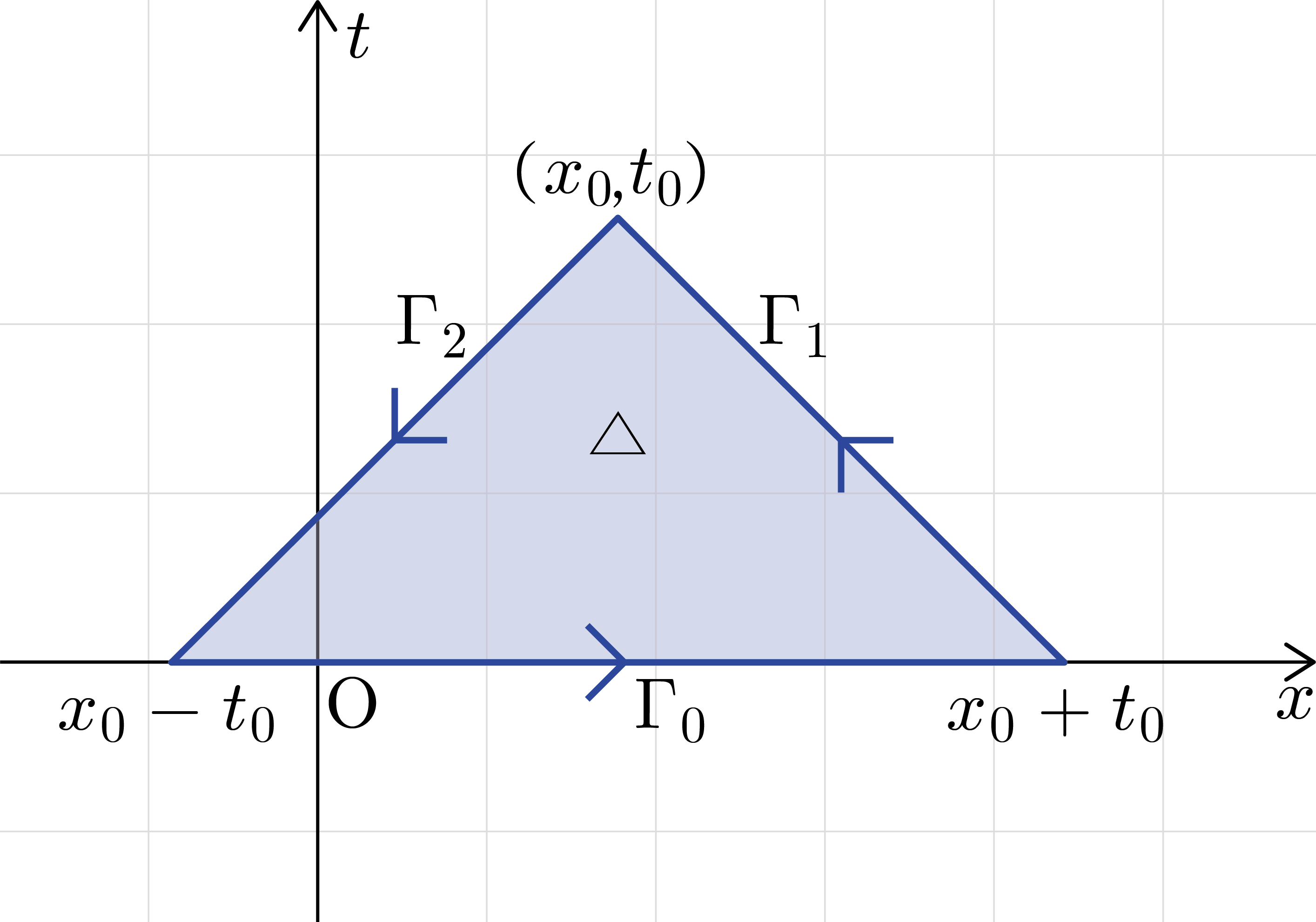} 
  \caption{The domain of dependence of the point $(x_0, t_0)$}
  \label{Fig : The domain of dependence of the point (x_0, t_0)}
  \end{figure}
  
  \noindent Note that
  \begin{equation} \label{Thm: IVP nonhomo wave eq. w/ spd in 1-dim - 2}
      \frac{dx}{dt} = -1 \quad \text{on~} \omega_1
      \qquad \text{and} \qquad
      \frac{dx}{dt} = 1 \quad \text{on~} \omega_2.
  \end{equation}
  By Theorem \ref{Thm : Green's Theorem with specular derivatives} and \eqref{Thm: IVP nonhomo wave eq. w/ spd in 1-dim - 2}, we find that 
  \begin{align*}
    \iint_{\bigtriangleup} \left( \frac{\partial u_x}{\partial^S x} - \frac{\partial u_t}{\partial^S t} \right) ~dx dt
    &= \int_{\partial\bigtriangleup} u_x ~dt + u_t ~dx \\
    &= \int_{\Gamma_0} u_x ~dt + u_t ~dx + \int_{\Gamma_1} u_x ~dt + u_t ~dx + \int_{\Gamma_2} u_x ~dt + u_t ~dx \\
    &= \int_{\Gamma_0} u_x ~dt + u_t ~dx - \left( \int_{\Gamma_1} u_x ~dx + u_t ~dt \right)+ \int_{\Gamma_2} u_x ~dx + u_t ~dt\\
    &= \int_{x_0 - t_0}^{x_0 + t_0} u_t(x, 0)~dx  -\int_{\Gamma_1}~du +  \int_{\Gamma_2}~du \\
    &= \int_{x_0 - t_0}^{x_0 + t_0} u_t(x, 0)~dx  - \left[ u(x_0, t_0) - u(x_0 + t_0, 0) \right] + \left[ u(x_0 - t_0, 0) - u(x_0, t_0) \right] \\
    &= \int_{x_0 - t_0}^{x_0 + t_0} \psi(x)~dx - \left[ u(x_0, t_0) - \varphi(x_0 + t_0) \right] + \left[ \varphi(x_0 - t_0) - u(x_0, t_0) \right] \\
    &= -2u(x_0, t_0) + \varphi(x_0 + t_0) + \varphi(x_0 - t_0) + \int_{x_0 - t_0}^{x_0 + t_0} \psi(x)~dx,
  \end{align*}
  which implies that 
  \begin{equation*} 
    u(x_0, t_0) = \frac{1}{2}\left[\varphi(x_0 + t_0) + \varphi(x_0 - t_0) \right]+ \frac{1}{2}\int_{x_0 - t_0}^{x_0 + t_0} \psi(x)~dx - \frac{1}{2}\iint_{\bigtriangleup} \left( \frac{\partial u_x}{\partial^S x} - \frac{\partial u_t}{\partial^S t} \right) ~dx dt .
  \end{equation*}
  Combining this result with \eqref{Thm: IVP nonhomo wave eq. w/ spd in 1-dim - 1}, we have 
  \begin{align*}
    u(x_0, t_0)  
    &= \frac{1}{2}\left[\varphi(x_0 + t_0) + \varphi(x_0 - t_0) \right] + \frac{1}{2}\int_{x_0 - t_0}^{x_0 + t_0} \psi(x)~dx + \frac{1}{2}\iint_{\bigtriangleup} f(x, t) ~dx dt \\
    &= \frac{1}{2}\left[\varphi(x_0 + t_0) + \varphi(x_0 - t_0) \right] + \frac{1}{2}\int_{x_0 - t_0}^{x_0 + t_0} \psi(x)~dx + \frac{1}{2}\int_{0}^{t}\int_{x_0-(t_0 - t)}^{x_0 + (t_0 - t)} f(x, t) ~dx dt,
  \end{align*}
  as required.
\end{proof}

\begin{remark}
  Note that the hypothesis \ref{(H)} can be dropped in solving the homogeneous wave equation with specular derivatives, which is more general case than \eqref{PDE: IVP homo wave eq. w/ spd in 1-dim}. 
  In other words, the solution \eqref{Sol : IVP nonhomo wave eq. w/ spd} also solve the equation \eqref{PDE: IVP homo wave eq. w/ spd in 1-dim} without \ref{(H)}.
\end{remark}

Unfortunately, the existence of the solution of the initial value problem for the nonhomogeneous wave equation with specular derivatives on $\mathbb{R}$ is not guaranteed.
Before we provide a counterexample, we find a suitable force which is in $S^0$-class.
For a constant $\lambda \in \mathbb{R}$, the \emph{exponential linear unit} $\operatorname{ELU}_{\lambda}:\mathbb{R} \to \mathbb{R}$ is defined by
\begin{equation} \label{Fnc : ELU}
  \operatorname{ELU}_{\lambda}(x) =
\begin{cases} 
x & \text{if } x \geq 0,\\ 
\lambda (e^x - 1) & \text{if } x < 0,
\end{cases}
\end{equation}
which is in $S^1(\mathbb{R})$ if $\lambda \neq 1$ and is in $S^2(\mathbb{R})$ if $\lambda = 1$.
We are interested in the case $\lambda = 1$; in this case, we have 
\begin{equation*} 
  \operatorname{(ELU_1)}^{\spd\spd}(x) =
  \begin{cases} 
  0 & \text{if } x > 0,\\ 
  A(0, 1) & \text{if } x = 0,\\ 
  e^x & \text{if } x < 0,
  \end{cases}
\end{equation*}
which is in $S^0(\mathbb{R})$.

Now, we provide a counterexample showing not that \eqref{PDE: IVP nonhomo wave eq. w/ spd in 1-dim} always has a solution.

\begin{example}
  Consider 
  \begin{equation} \label{Ex: IVP nonhomo wave eq. w/ spd in 1-dim}
    \begin{cases} 
      \displaystyle
      \partial^S_t u_t - \partial^S_x u_x = f(x, t), & (x, t) \in \mathbb{R} \times (0, \infty),\\ 
      u(x, 0) = \varphi(x), & x \in \mathbb{R},\\ 
      u_t(x, 0) = 0, & x \in \mathbb{R} ,
    \end{cases}
  \end{equation}
  where $f$ is defined by 
  \begin{equation} \label{Ex: IVP nonhomo wave eq. w/ spd in 1-dim - 1}
    f(x, t) =
    \begin{cases} 
    -1 & \text{if } x-t > 0, x+t > 0,\\ 
    A(-1, 0) & \text{if } x-t = 0, x+t > 0,\\ 
    0 & \text{if } x-t < 0, x+t > 0,\\ 
    A(0, 1) & \text{if } x-t < 0, x+t = 0,\\ 
    1 & \text{if } x-t < 0, x+t < 0,
    \end{cases}
  \end{equation}
  and $\varphi$ is defined by 
  \begin{equation*}
    \varphi(x) =
    \begin{cases} 
      \displaystyle \frac{1}{2}x^2 + 2x & \text{if } x \geq 0,\\[0.2cm] 
      \displaystyle 2e^x -\frac{1}{2}x^2 - 2 & \text{if } x < 0.
    \end{cases}
  \end{equation*}

  For convenience, decompose the domain $\mathbb{R}^{n} \times (0, \infty)$ as follows:
  \begin{align*}
    \Omega_1 & := \left\{ (x, t) : x-t > 0, x+t > 0 \right\}, \\
    \Gamma_1 & := \left\{ (x, t) : x-t = 0, x+t > 0 \right\}, \\
    \Omega_2 & := \left\{ (x, t) : x-t < 0, x+t > 0 \right\}, \\
    \Gamma_2 & := \left\{ (x, t) : x-t < 0, x+t = 0 \right\}, \\
    \Omega_3 & := \left\{ (x, t) : x-t < 0, x+t < 0 \right\},
  \end{align*}
  which are illustrated as in Figure \ref{Fig : The decomposition of R X (0, infty)}.

  \begin{figure}[H] 
  \centering 
  \includegraphics[scale=1]{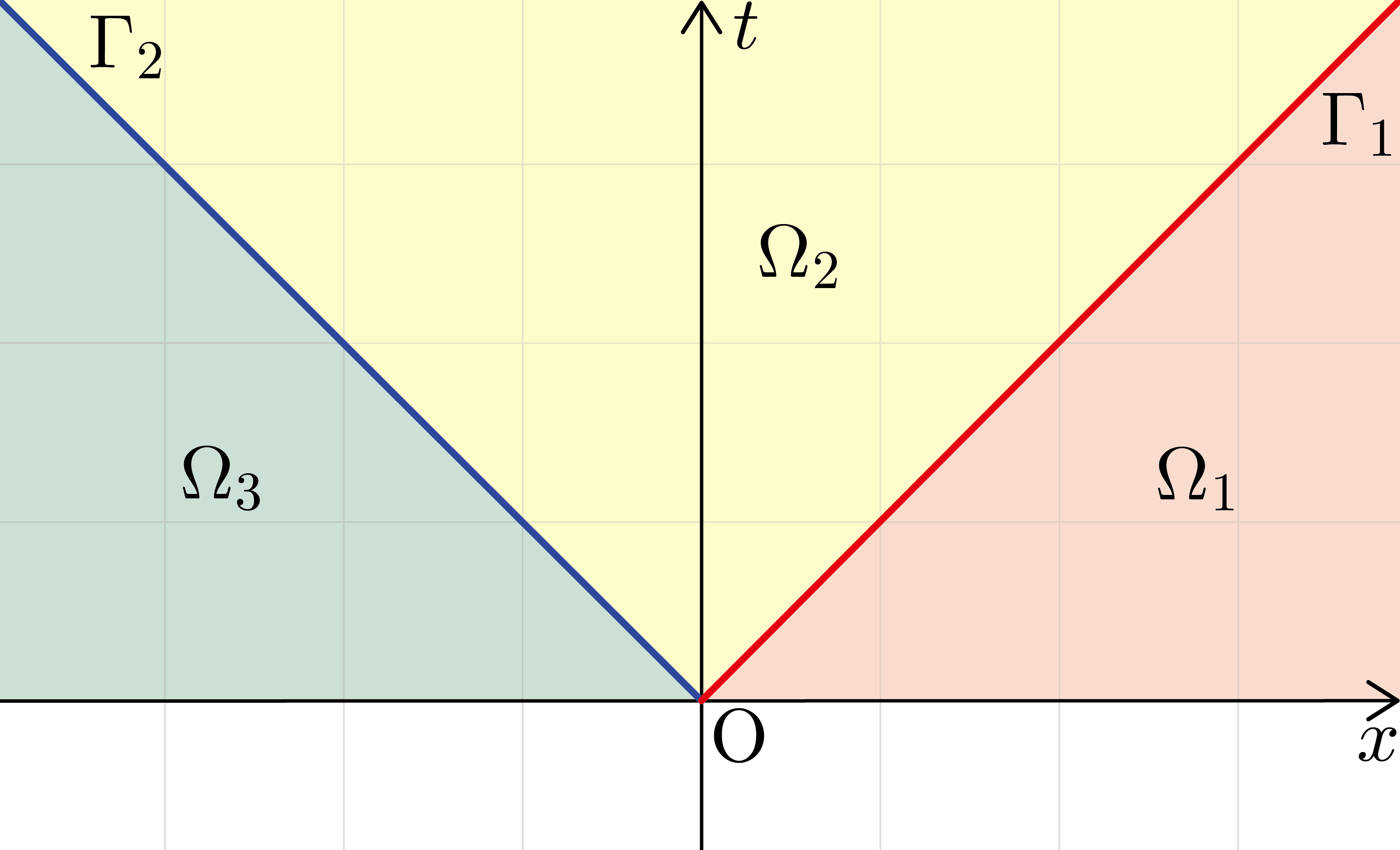} 
  \caption{Decomposition of $\mathbb{R} \times (0, \infty)$}
  \label{Fig : The decomposition of R X (0, infty)}
  \end{figure}

  First, we calculate the solution $u$ by applying the formula \eqref{Sol : IVP nonhomo wave eq. w/ spd}.
  The first term of \eqref{Sol : IVP nonhomo wave eq. w/ spd} can be reduced to 
  \begin{align*}
    \frac{1}{2} [\varphi(x + t) + \varphi(x - t)] 
    %&= 
    %\begin{cases} 
    %\frac{1}{2}\left[ (x + t)^2 + 2(x + t) + (x - t)^2 + 2(x - t) \right] & \text{if } (x, %t) \in \Omega_1 \cup \Gamma_1,\\ 
    %\frac{1}{2}\left[ 2e^{x + t}-(x + t)^2 - 2 \right] & \text{if } (x, t) \in \Omega_2 %\cup \Gamma_2,\\ 
    %\frac{1}{2}\left[ 2e^{x + t}-(x + t)^2 - 2 + 2e^{x - t}-(x - t)^2 - 2 \right] & \text%{if } (x, t) \in \Omega_3 
    %\end{cases} \\
    &=   
    \begin{cases} 
      \displaystyle \frac{1}{2}x^2 + \frac{1}{2}t^2 + 2x  & \text{if } (x, t) \in \Omega_1 \cup \Gamma_1,\\[0.2cm] 
      \displaystyle e^{x - t} + xt + x + t - 1 & \text{if } (x, t) \in \Omega_2 \cup \Gamma_2,\\[0.2cm] 
      \displaystyle e^{x + t} + e^{x - t} - \frac{1}{2}x^2 - \frac{1}{2}t^2 - 2 & \text{if } (x, t) \in \Omega_3, 
    \end{cases}
  \end{align*}
  and the second term of \eqref{Sol : IVP nonhomo wave eq. w/ spd} is reduced to zero.
  One can calculate the last term of \eqref{Sol : IVP nonhomo wave eq. w/ spd}, resulting in
  \begin{align*}
    \frac{1}{2} \int_{0}^{t} \int_{x-t+s}^{x+t-s} f(y, s) ~dy ds = 
    \begin{cases} 
      \displaystyle -\frac{1}{2}t^2 & \text{if } (x, t) \in \Omega_1 \cup \Gamma_1,\\[0.2cm]
      0 & \text{if } (x, t) \in \Omega_2 \cup \Gamma_2,\\[0.2cm]
      \displaystyle \frac{1}{2}t^2 & \text{if } (x, t) \in \Omega_3.
    \end{cases}
  \end{align*}
  Summing these results, we obtain that 
  \begin{equation*} 
    u(x, t) = 
    \begin{cases} 
      \displaystyle \frac{1}{2}x^2 + 2x & \text{if } (x, t) \in \Omega_1 \cup \Gamma_1,\\[0.2cm]
      e^{x - t} + xt + x + t - 1 & \text{if } (x, t) \in \Omega_2 \cup \Gamma_2,\\ 
      \displaystyle e^{x + t} + e^{x - t} -\frac{1}{2}x^2 - 2 & \text{if } (x, t) \in \Omega_3,
    \end{cases}
  \end{equation*}
  which can be represented as
  \begin{equation*}
    u(x, t) = \frac{1}{2}g(x, t) + \operatorname{ELU}_1(x + t) + \operatorname{ELU}_1(x - t),
  \end{equation*}
  where the function $g$ is defined by 
  \begin{equation*}
   g(x, t)=
  \begin{cases} 
   x^2 & \text{if } (x, t) \in \Omega_1 \cup \Gamma_1,\\
   2xt & \text{if } (x, t) \in \Omega_2 \cup \Gamma_2,\\
   -x^2 & \text{if } (x, t) \in \Omega_3,
  \end{cases}
  \end{equation*}
  and $\operatorname{ELU}_1$ is defined as in \eqref{Fnc : ELU}.
  It is clear that $u$ satisfies the initial condition.

  Second, $u$ solves \eqref{Ex: IVP nonhomo wave eq. w/ spd in 1-dim}.
  Indeed, we find that 
  \begin{align*}
    u_x(x, t) 
    &=
    \begin{cases} 
      \displaystyle x + 2
      & \text{if } (x, t) \in \Omega_1 \cup \Gamma_1,\\
      \displaystyle e^{x - t} + t + 1
      & \text{if } (x, t) \in \Omega_2 \cup \Gamma_2,\\ 
      \displaystyle e^{x + t} + e^{x - t} -x
      & \text{if } (x, t) \in \Omega_3,
    \end{cases}
  \end{align*}
  and 
  \begin{align*}
    u_t(x, t) 
    &=
    \begin{cases} 
      \displaystyle 0
      & \text{if } (x, t) \in \Omega_1 \cup \Gamma_1,\\
      \displaystyle -e^{x - t} + x + 1
      & \text{if } (x, t) \in \Omega_2 \cup \Gamma_2,\\ 
      \displaystyle e^{x + t} -e^{x - t}
      & \text{if } (x, t) \in \Omega_3.
    \end{cases}
  \end{align*}
  Obviously, $u_t(x, 0)$ meets the initial condition.
  Also, we find that 
  \begin{align*}
    \partial^S_tu_x(x, t) &=
    \begin{cases} 
      0 & \text{if } (x, t) \in \Omega_1 \cup \Gamma_1,\\
      -e^{x-t} + 1 & \text{if } (x, t) \in \Omega_2 \cup \Gamma_2,\\ 
      e^{x + t} -e^{x - t} & \text{if } (x, t) \in \Omega_3
    \end{cases}\\
    &= \partial^S_xu_t(x, t),
  \end{align*}
  which implies that $\partial^S_tu_x(x, t)$ and $\partial^S_xu_t(x, t)$ are continuous in $\mathbb{R} \times (0, \infty)$.
  To check $\partial^S_t u_t - \partial^S_x u_x = f(x, t)$, we calculate that 
  \begin{align*}
    \partial^S_t u_t - \partial^S_x u_x &=
    \begin{cases} 
      0 - 1 & \text{if } (x, t) \in \Omega_1,\\
      A(-1, 0) & \text{if } (x, t) \in \Gamma_1,\\
      e^{x - t} - e^{x - t} & \text{if } (x, t) \in \Omega_2,\\ 
      A(0, 1) & \text{if } (x, t) \in \Gamma_2,\\ 
      \left(e^{x + t} + e^{x - t} \right) - \left(e^{x + t} + e^{x - t} - 1 \right) & \text{if } (x, t) \in \Omega_3,
    \end{cases}
  \end{align*} 
  which is equal with the given function $f$ as in \eqref{Ex: IVP nonhomo wave eq. w/ spd in 1-dim - 1}.

  However, $u$ has discontinuity on $\Gamma_1 \cup \Gamma_2$ so that $u$ is not in $S^2$-class, even if $u$ solves \eqref{Ex: IVP nonhomo wave eq. w/ spd in 1-dim} and meets the initial conditions.
  For example, if $x = t$, then 
  \begin{equation*} 
    u(x, t) = 
    \begin{cases} 
      \displaystyle \frac{1}{2}x^2 + 2x & \text{if } (x, t) \in \Omega_1 \cup \Gamma_1,\\[0.2cm]
      x^2 + 2x & \text{if } (x, t) \in \Omega_2 \cup \Gamma_2,\\ 
      \displaystyle e^{2x} -\frac{1}{2}x^2 - 1 & \text{if } (x, t) \in \Omega_3,
    \end{cases}
  \end{equation*}
  which implies that $u$ has discontinuity on $\Gamma_1$.
\end{example}

\bibliographystyle{abbrv}
\bibliography{Bibliography/The_wave_equation_with_specular_derivatives}{}

\end{document}